\documentclass[a4paper,12pt,reqno]{amsart}
\usepackage{amssymb,amsmath,amsthm}
\usepackage{graphicx}

\DeclareGraphicsExtensions{.bmp,.png,.pdf,.jpg,}

\usepackage{bm}
\usepackage{verbatim}
\usepackage{color}
\usepackage{subcaption}
\def\R{{\mathbb R}}

\def\a{\alpha}
\def\b{\beta}
\def\t{\theta}

\def\s{\sigma}
\def\t{\theta}
\def\l{\lambda}

\def\e{\varepsilon}
\def\v{\varphi}

\def\mc{\mathcal}

\def\ua{\uparrow}
\def\da{\downarrow}

\numberwithin{equation}{section}

\theoremstyle{definition}

\theoremstyle{plain}
\newtheorem{theorem}{Theorem}[section]
\newtheorem{proposition}{Proposition}[section]

\theoremstyle{definition}

\begin{document}

\title[Nodal solutions of weighted indefinite problems]
{Nodal solutions of weighted indefinite problems}

\author{M. Fencl}
\address{Department of Mathematics and NTIS, Faculty of Applied Sciences,
University of West Bohemia, Univerzitn\'i 8, 30100, Plzen, Czech Republic}
\email{fenclm37@ntis.zcu.cz}
\author{J. L\'{o}pez-G\'{o}mez}
\address{Institute of Inter-disciplinar Mathematics (IMI) and Department of Analysis and Applied Mathematics, Complutense University of Madrid, Madrid 28040, Spain}
\email{Lopez$\_$Gomez@mat.ucm.es}

\maketitle

\centerline{\it This paper is dedicated to M. Hieber}

\centerline{\it at the occasion of his 60th birthday}

\centerline{\it mit Wertsch\"{a}tzung und Freundschaft}

\begin{abstract}
This paper analyzes the structure of the set of nodal solutions of a class of one-dimensional
superlinear indefinite boundary values problems with an indefinite weight functions in front of
the spectral parameter. Quite astonishingly, the associated high order eigenvalues might not
be concave as it is the lowest one. As a consequence, in many circumstances the nodal solutions
can bifurcate from three or even four bifurcation points from the trivial solution. This paper combines
analytical and numerical tools. The analysis carried over on it is a paradigm of how mathematical analysis aids the numerical study of a problem, whereas simultaneously the numerical study confirms and illuminate the analysis.
\vspace{0.1cm}

\vspace{0.1cm}

\vspace{0.1cm}

\noindent Partially supported by the Research Grant  PGC2018-097104-B-I00 of the Spanish Ministry of Science, Innovation and Universities,  and the Institute of Inter-disciplinar Mathematics (IMI) of Complutense University. M. Fencl has been supported by the project SGS-2019-010 of the University of West Bohemia, the project 18-03253S of the Grant Agency of the Czech Republic and the project LO1506 of the Czech Ministry of Education, Youth and Sport.
\end{abstract}

\section{Introduction}

\noindent In this paper we analyze the nodal solutions of the one-dimensional nonlinear weighted boundary value problem
\begin{equation}
\label{1.1}
  \left\{ \begin{array}{l} -u''-\mu u = \l m(x) u-a(x) u^2  \quad \hbox{in} \;\; (0,1), \\[1ex]
  u(0)=u(1)=0, \end{array}\right.
\end{equation}
where $a, m \in \mathcal{C}[0,1]$ are functions that change sign in $(0,1)$  and $\l, \mu \in\R$ are regarded as bifurcation parameters. More precisely, $\l$ is the primary parameter, and $\mu$ the secondary one. All the numerical experiments carried out in this paper have been implemented in the special case when
\begin{equation}
\label{1.2}
  a(x):=\left\{ \begin{array}{ll} -0.2 \sin\left( \frac{\pi}{0.2}(0.2-x)\right) & \quad
  \hbox{if}\;\; 0\leq x \leq 0.2, \\[2ex] \sin\left( \frac{\pi}{0.6}(x-0.2)\right)
  & \quad   \hbox{if}\;\; 0.2< x \leq 0.8, \\[2ex] -0.2 \sin\left( \frac{\pi}{0.2}(x-0.8)\right) & \quad
   \hbox{if}\;\; 0.8 < x \leq 1,\end{array}\right.
\end{equation}
because this is the weight function $a(x)$ considered by L\'{o}pez-G\'{o}mez and Molina-Meyer in  \cite{LGMM} to compute the global bifurcation diagrams of positive solutions there in. In this paper we pay
a very special attention to the particular, but very interesting, case when
$$
  m(x)= \sin (n\pi x)
$$
for some integer $n\geq 2$.
\par
Up to the best of our knowledge, this is the first paper where the problem of the existence and the structure of the nodal solutions of a weighted superlinear indefinite problem is addressed when $m(x)$ changes of sign. The existence results of large solutions of Mawhin, Papini and Zanolin \cite{MPZ} required $m\equiv 1$, as well as the results of  L\'{o}pez-G\'{o}mez, Tellini and Zanolin \cite{LGTZ}, where the attention was focused on the problem of ascertaining the structure of the set of positive solutions. Most of the available results on nodal solutions  dealt with the special cases when $m\equiv 1$, $\mu=0$ and $a(x)$ is a positive function with $\min_{[0,1]}a>0$ (see Rabinowitz \cite{Ra1,Ra2,Ra3}), or with the degenerate case when $a(x)$ is a continuous positive function
such that $a^{-1}(0)=[\a,\b]\subset (0,1)$ (see L\'{o}pez-G\'{o}mez and Rabinowitz \cite{LGR1,LGR2,LGR3}, and L\'{o}pez-G\'{o}mez, Molina-Meyer and Rabinowitz  \cite{LGMMR}). In strong contrast with the classical cases when  $\min_{[0,1]}a>0$, in the degenerate case when $a\geq 0$ with $a^{-1}(0)=[\a,\b]\subset (0,1)$ the set of nodal solutions might consist of two, or even more, components, depending on the nature of the weight function $a(x)$ (see \cite{LGMMR} and \cite{LGR3} for any further required details).
Nevertheless, as for the special choice $a(x)$ given by \eqref{1.2}, $a(x)$ is negative in the intervals $(0,0.2)$ and $(0.8,1)$, while it is positive in the central interval $(0.2,0.8)$, this is the first time that the problem of analyzing the structure of the nodal solutions in this type of superlinear indefinite problems is addressed.
\par
A natural strategy for constructing the nodal solutions of \eqref{1.1} with $n\geq 0$ interior zeroes, or nodes, consists in linearizing \eqref{1.1} at the trivial solution, $u=0$, and then searching  for the eigenvalues of the linearization having an associated eigenfunction with exactly $n$ interior nodes in $(0,1)$, for as these values of the parameters will provide us, through the local bifurcation theorem of Crandall and Rabinowitz \cite{CR}, with all the small nodal solutions of \eqref{1.1} bifurcating from $u=0$. This strategy provides us  in a rather natural way with the linear weighted eigenvalue problem
\begin{equation}
\label{1.3}
  \left\{ \begin{array}{l} -\v''-\mu \v - \l m(x) \v = \sigma \v   \quad \hbox{in} \;\; (0,1), \\[1ex]
  \v(0)=\v(1)=0. \end{array}\right.
\end{equation}
By the Sturm--Liouville theory, the problem \eqref{1.3} has  a sequence of eigenvalues
\[
  \Sigma_n(\l,\mu):=  \s_n[-D^2-\mu-\l m(x);(0,1)],\qquad n\geq 1,
\]
which are algebraically simple. Moreover, associated with each of them there is an  eigenfunction, $\v_n$, with $\v_n'(0)>0$, unique up to a multiplicative constant, with  exactly $n-1$  interior nodes, necessarily simple,  in $(0,1)$. By uniqueness,
\begin{equation}
\label{1.4}
  \Sigma_n(\l,\mu):=  \s_n[-D^2-\l m(x);(0,1)]-\mu,\qquad n\geq 1.
\end{equation}
It turns out that the set of all the possible bifurcation points from $u=0$ to solutions of \eqref{1.1} with $n-1$ interior zeroes are provided by the values of $\l$ and $\mu$ for which
\[
  \Sigma_n(\l,\mu)=0.
\]
So, the huge interest in analyzing them. Throughout this paper, we will denote
\begin{equation}
\label{1.5}
  \Sigma_n(\l):=\Sigma(\l,0)=  \s_n[-D^2-\l m(x);(0,1)],\qquad n\geq 1.
\end{equation}
Then,
\[
  \Sigma_n(\l,\mu)=\Sigma_n(\l)-\mu
\]
and $\Sigma_n(0)=(n\pi)^2$ for all $n\geq 1$. Based on a classical result of Kato \cite{KaB} on
perturbation from simple eigenvalues,  for every $n\geq 1$, $\Sigma_n(\l)$ is analytic in $\l\in\R$.
A proof of this can be easily accomplished from \cite[Ch. 9]{LG13} and Section 5 of
Ant\'{o}n and L\'{o}pez-G\'{o}mez \cite{ALG}, where the result was established when $n=1$. An extremely important property of $\Sigma_1(\l)$ is its strict concavity with respect to the
parameter $\l$ (see Berestycki, Nirenberg and Varadhan \cite{BNV}, Cano-Casanova and L\'{o}pez-G\'{o}mez \cite{CCLG} and Chapter 9 of \cite{LG13}). According to it, $\Sigma_1'(\l)>0$ for all $\l< 0$, $\Sigma_1'(0)=0$, $\Sigma_1'(\l)<0$ for all $\l>0$, and
\begin{equation}
\label{1.6}
  \Sigma_1''(\l)<0 \quad \hbox{for all}\;\; \l\in\R.
\end{equation}
Since $\Sigma_1(0)=\pi^2$, this property entails that, for every $\mu<\pi^2$, $\Sigma_1^{-1}(\mu)$
consists of two values of $\l$,
$$
  \l_-\equiv \l_-(\mu) <0< \l_+\equiv \l_+(\mu),
$$
which are the unique bifurcation values to positive solutions from $u=0$ of \eqref{1.1} (see L\'{o}pez-G\'{o}mez and Molina-Meyer \cite{LGMM}).
Even dealing with general second order elliptic operators under general mixed boundary conditions of
non-classical type, the strict concavity of $\Sigma_1(\l)$  relies on the strong ellipticity of the
elliptic operator (see, e.g., Chapter 8 of \cite{LG13}).
\par
For analytic semigroups the spectral mapping theorem holds (see, e.g., \cite{ABHN,AN}), i.e,
\[
  \s(e^{D^2+\l m})\setminus\{0\}=e^{-\s(-D^2-\l m)}=\left\{ e^{-\sigma_n(-D^2-\l m;(0,1))}\;:\;n\geq 1\right\}.
\]
Thus, the spectral radius of the associated semigroup is given through the formula
\[
  \varrho(\l):= \mathrm{spr\,}(e^{D^2+\l m})= e^{-\sigma_1(-D^2-\l m;(0,1))}=e^{-\Sigma_1(\lambda)},\qquad \l\in\R.
\]
Hence, $\varrho(\l)$ is logarithmically convex, which is a classical property going back to
Kato \cite{KaA}, because $-\Sigma_1(\lambda)$ is convex. Rather astonishingly, there are examples of weight functions $m(x)$ for which none of the remaining eigenvalues $\Sigma_n(\l)$, $n\geq 2$, is concave with respect to $\l$. Figure \ref{Fig1} shows one of these examples for the special choice $m(x)=\sin (2\pi x)$.
\begin{figure}[h!]
\centering
\includegraphics[width=1\linewidth]{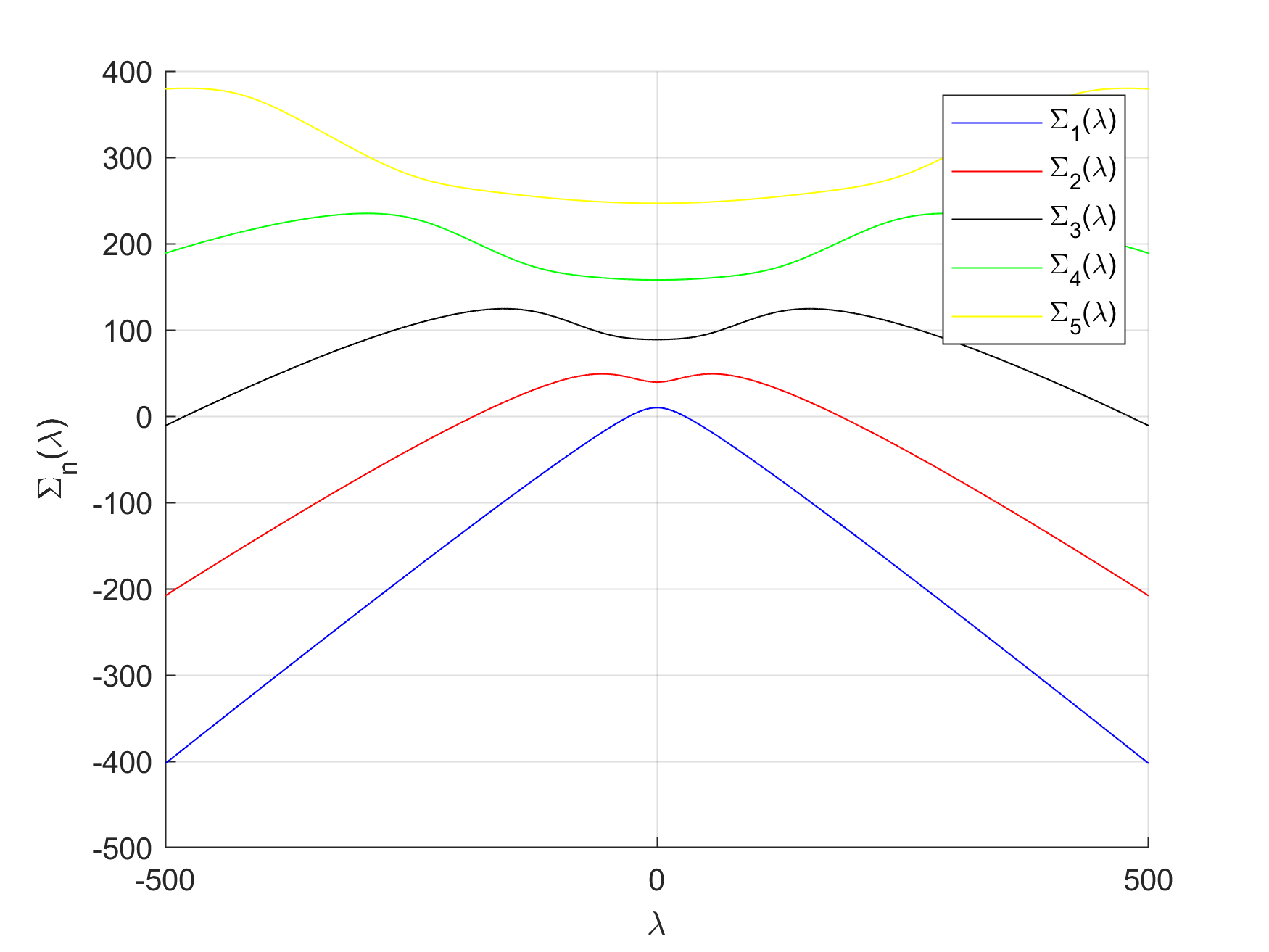}
\caption{The curves $\Sigma_n(\lambda)$ for $1\leq n\leq 5$ and $m(x)=\sin(2\pi x)$.}
\label{Fig1}
\end{figure}
In this case, $\Sigma_1(\l)$ is the unique
eigencurve which is concave, for as the remaining ones, $\Sigma_n(\l)$, $n\geq 2$, are far from concave. Indeed, all of them are symmetric functions of $\l$, with a quadratic local minimum at $\l=0$, as illustrated by Figure \ref{Fig1}. This fact has dramatic implications from the point of view of the structure of the set of nodal solutions of the problem \eqref{1.1}. Indeed, setting
\begin{equation}
\label{1.7}
  \mu_n = \max_{\l\in\R}\Sigma_n(\l),\qquad n\geq 1,
\end{equation}
it becomes apparent that $\mu_n>\Sigma_n(0)=(n\pi)^2$ for all $n\geq 1$ and, hence, for every $n\geq 2$ and any $\mu\in ((n\pi)^2,\mu_n)$, $\Sigma_n^{-1}(\mu)$ consists of two negative eigenvalues,
$\l_{-[1,n]}(\mu)<\l_{-[2,n]}(\mu)<0$, and two positive eigenvalues $0<\l_{+[2,n]}(\mu)<\l_{+[1,n]}(\mu)$ such that
$$
  0< \l_{+[2,n]}(\mu)=-\l_{-[2,n]}(\mu)<\l_{+[1,n]}(\mu)=-\l_{-[1,n]}(\mu).
$$
Therefore, for this range of $\mu$'s we expect that the solutions with $n-1$ interior nodes of
\eqref{1.1} will bifurcate from the trivial solution at each of the four values
\[
  \l =\l_{\pm[i,n]},\quad i=1,2.
\]
By simply having a look at Figure \ref{Fig1}, it is easily realized that
\[
  \l_{\pm[2,n]}((n\pi)^2)=0.
\]
Moreover,
\[
  \l_{-[1,n]}(\mu_n)=\l_{-[2,n]}(\mu_n)< 0< \l_{+[2,n]}(\mu_n)=\l_{+[1,n]}(\mu_n),
\]
at least for $n\in\{2,3,4,5\}$.
\par
As illustrated by Figure \ref{Fig2}, the number of eigencurves, $\Sigma_n(\l)$, $n\geq 2$, which are
concave in $\l$ might vary with the weight function $m(x)$. Indeed, when $m(x)=\sin(4\pi x)$, it turns out that not only $\Sigma_1(\l)$ but also $\Sigma_2(\l)$  is strictly concave, while the remaining eigencurves, $\Sigma_n(\l)$, with $n\geq 3$, are not concave. Similarly, when
$m(x)=\sin(6\pi x)$, then $\Sigma_j(\l)$ are concave for $j\in\{1,2,3\}$, while they are not concave for $j\geq 4$.
\par
\begin{figure}[h!]
\centering
\includegraphics[width=0.51\linewidth]{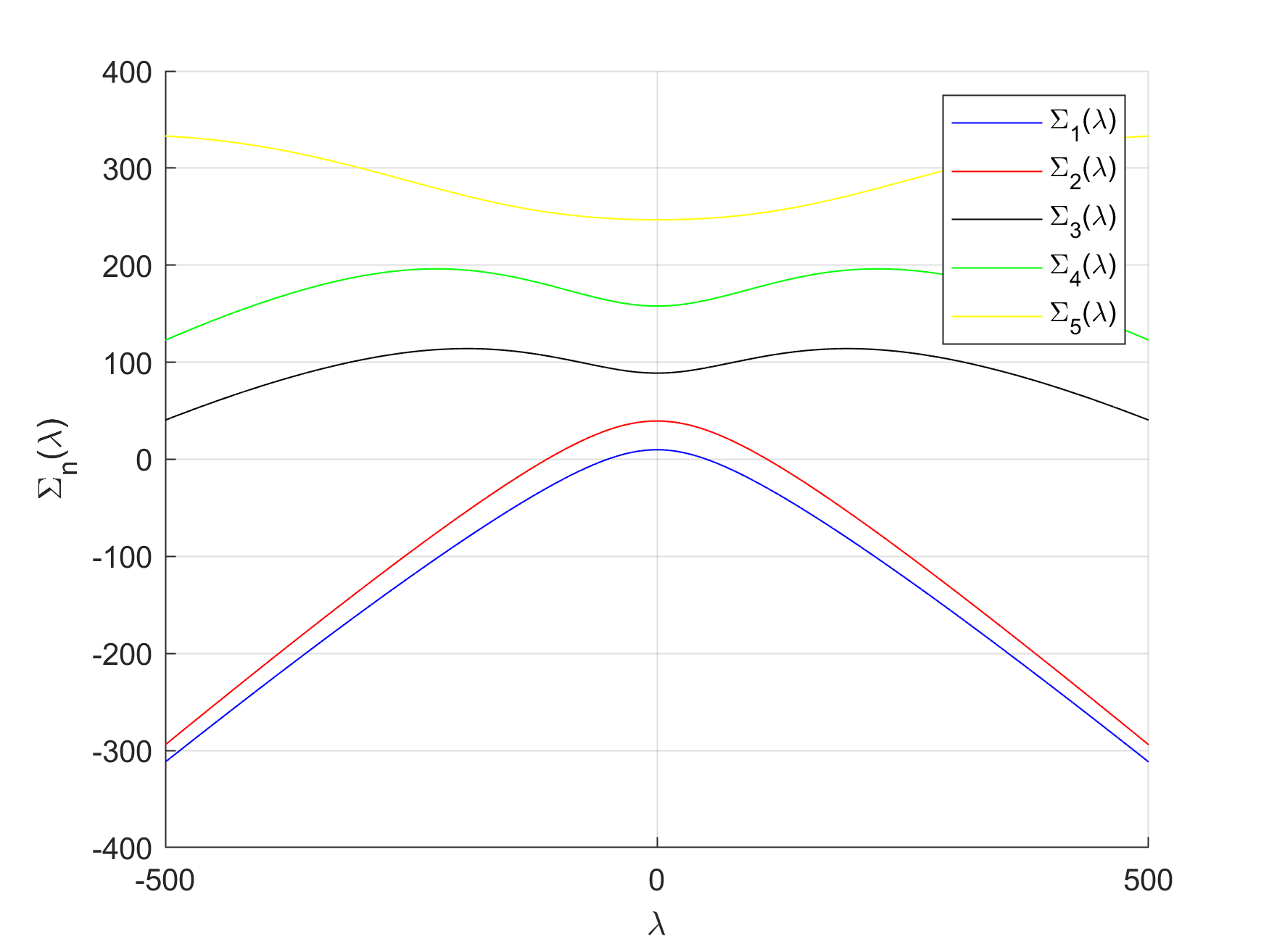}\hspace{-0.5cm} \includegraphics[width=0.51\linewidth]{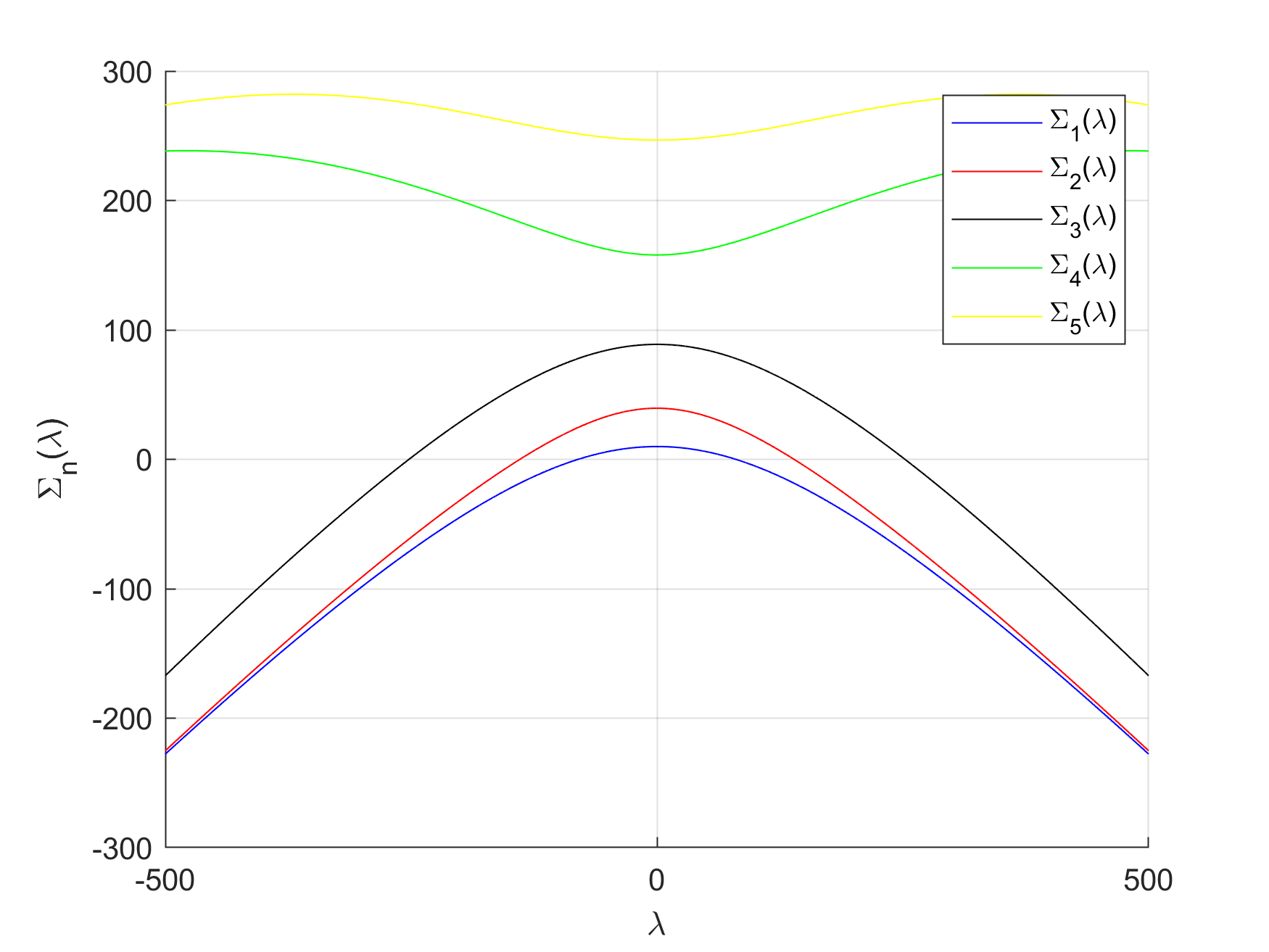}
\caption{The curves $\Sigma_n(\lambda)$ for $1\leq n\leq 5$ with $m(x)=\sin(n\pi x)$, $n=4, 6$.}
\label{Fig2}
\end{figure}
Quite astonishingly, as suggested by our numerical computations, the more wiggled is $m(x)$ the higher number of modes $\Sigma_n(\l)$ is concave. This astonishing feature might have some important implications in quantum mechanics.
\par
The distribution of this paper is as follows. Section 2 studies some global properties of the eigencurves $\Sigma_n(\l)$ for all $n\geq 2$ and analyzes their concavities in the special case when, for some $k\geq 1$,
\begin{equation}
\label{1.8}
  m(x)=\sin (2k\pi x),\qquad x\in [0,1].
\end{equation}
Section 3 provides us some global bifurcation diagramas of nodal solutions
of \eqref{1.1} with one and two interior nodes, which are superimposed to the global bifurcation
diagrams of  positive solutions of L\'{o}pez-G\'{o}mez and Molina-Meyer \cite{LGMM}. Finally, in Section 4 we describe, very shortly, the numericical schemes used to get the global bifurcation diagrams of Section 3.

\section{Some global properties of the nodal eigencurves $\Sigma_n(\l)$}

\noindent Throughout this paper, for any given $r, s\in \R$ with $r<s$ and every continuous function
$q \in \mathcal{C}[r,s]$, we denote by $\s_n[-D^2+q(x);(r,s)]$, $n\geq 1$, the $n$-th
 eigenvalue of the eigenvalue problem
\begin{equation}
\label{2.1}
  \left\{ \begin{array}{l} -\v''+q(x) \v = \sigma \v   \quad \hbox{in} \;\; (r,s), \\[1ex]
  \v(r)=\v(s)=0. \end{array}\right.
\end{equation}
The next properties are well known (see, e.g., \cite{BGH}):
\begin{enumerate}
\item[ i)] \emph{Monotonicity of $\s_n$ with respect to $q(x)$:} If $q, \tilde q \in \mathcal{C}[r,s]$ satisfy $q\lneq \tilde q$, then
$$
  \s_n[-D^2+q(x);(r,s)]<\s_n[-D^2+\tilde q;(r,s)]\quad \hbox{for all}\;\; n\geq 1.
$$
\item[ ii)]  \emph{Monotonicity of $\s_n$ with respect to the interval:} If $[\a,\b]\subset (r,s)$, then
$$
  \s_n[-D^2+q;(r,s)]<\s_n[-D^2+q;(\a,\b)]\quad \hbox{for all}\;\; n\geq 1.
$$
\end{enumerate}
Based on these properties, as suggested by Figures \ref{Fig1} and \ref{Fig2}, the next result holds.

\begin{proposition}
\label{pr2.1}
Suppose that there exist $x_\pm\in(0,1)$ such that $\pm m(x_\pm)>0$, i.e., $m(x)$ changes the sign in $(0,1)$. Then, for every $n\geq 1$,
\begin{equation}
\label{2.2}
\lim_{\l\da -\infty}\Sigma_n(\l)=-\infty,\qquad \lim_{\l\ua \infty}\Sigma_n(\l)=-\infty.
\end{equation}
\end{proposition}
\begin{proof}
Consider a sufficiently small $\e>0$ such that
$$
  J_\e:=[x_+-\e,x_++\e]\subset (0,1),\qquad \min_{J_\e}m=m_L>0.
$$
Then, by the monotonicity properties of $\Sigma_n$, for every $\l>0$ and
$n\geq 1$, we have that
\begin{align*}
  \Sigma_n(\l) & = \s_n[-D^2-\l m(x);(0,1)]< \s_n[-D^2-\l m(x);J_\e]\\ & <
  \s_n[-D^2-\l m_L;J_\e]=\s_n[-D^2;J_\e]-\l m_L = \left(\frac{n\pi}{2\e}\right)^2-\l m_L.
\end{align*}
Thus, letting $\l\ua \infty$, the second relation of \eqref{2.2} holds. The first one follows by applying this result to the weight function $-m(x)$. This ends the proof.
\end{proof}

The fact that all the eigencurves plotted in Figures \ref{Fig1} and \ref{Fig2} are symmetric about the
ordinate axis is a direct consequence of the next general result, because
\[
  \sin\left(2k\pi(1-x)\right)=-\sin (2k\pi x),
\]
for all integer $k\geq 1$ and $x\in [0,1]$.

\begin{proposition}
\label{pr2.2}
Suppose that $m\neq 0$ is a continuous function in $[0,1]$ such that
\begin{equation}
\label{2.3}
  m(1-x)=-m(x)\quad \hbox{for all}\;\; x\in [0,1];
\end{equation}
this holds under condition $(\ref{1.8})$. Then, $\Sigma_n(-\l)=\Sigma_n(\l)$ for all  $\l\in\R$ and any integer $n\geq 1$. In particular,
\begin{equation}
\label{2.4}
   \dot{ \Sigma}_n(0)=0\quad \hbox{for all} \;\; n\geq 1,
\end{equation}
where we are denoting $\dot{ \Sigma}_n=\frac{d\Sigma_n}{d\l}$.
\end{proposition}
\begin{proof}
Since $m\neq 0$, either there exists $x_+\in (0,1)$ such that $m(x_+)>0$, or $m(x_-)<0$ for some
$x_-\in (0,1)$. Suppose the first alternative occurs. Then, by \eqref{2.3}, we also have that
$$
  m(1-x_+)=-m(x_+)<0
$$
and hence, $m(x)$  changes the sign in $(0,1)$. In particular, \eqref{2.2} holds.
\par
Pick an integer $n\geq 1$, a real number $\l$, and let $\phi_n$ be an eigenfunction associated to
$\Sigma_n(\l)$. Then, $\phi_n$ possesses $n-1$ zeros in $(0,1)$, $\phi_n(0)=\phi_n(1)=0$, and
\begin{equation*}
  -\phi_n''(x)=\l m(x) \phi_n(x)+\Sigma_n(\l)\phi_n(x)
\end{equation*}
for all $x\in (0,1)$. Thus, setting
$$
   \psi_n(x):=\phi_n(1-x),\qquad x\in [0,1],
$$
it is easily seen that
$$
   \psi_n'(x):=-\phi_n'(1-x),\quad \psi_n''(x)=\phi_n''(1-x), \qquad x\in [0,1],
$$
and hence, for every $x\in (0,1)$,
\begin{align*}
  -\psi_n''(x)  =-\phi_n''(1-x) & = \l m(1-x) \phi_n(1-x)+\Sigma_n(\l)\phi_n(1-x)\\
  & =  \l m(1-x) \psi_n(x)+\Sigma_n(\l)\psi_n(x)\\ &  = -\l m(x) \psi_n(x)+\Sigma_n(\l)\psi_n(x).
\end{align*}
Consequently, $\psi_n(x)$ is an eigenfunction associated to $-D^2+\l m(x)$ with $n-1$ interior zeros.
Therefore, by the uniqueness of $\Sigma_n$, it becomes apparent that
$$
  \Sigma_n(-\l)=\Sigma_n(\l)\quad \hbox{for all}\;\; \l\in\R.
$$
Since $\Sigma_n(\l)$ is an analytic function of $\l$, necessarily $\dot{\Sigma}_n(0)=0$.
This ends the proof.
\end{proof}

By having a glance at Figure \ref{Fig3}, it is easily realized that the function $\Sigma_n(\l)$ might not be an even function of $\l$ if condition \eqref{2.3} fails.
\par
\begin{figure}[h!]
	\centering
	\includegraphics[width=0.51\linewidth]{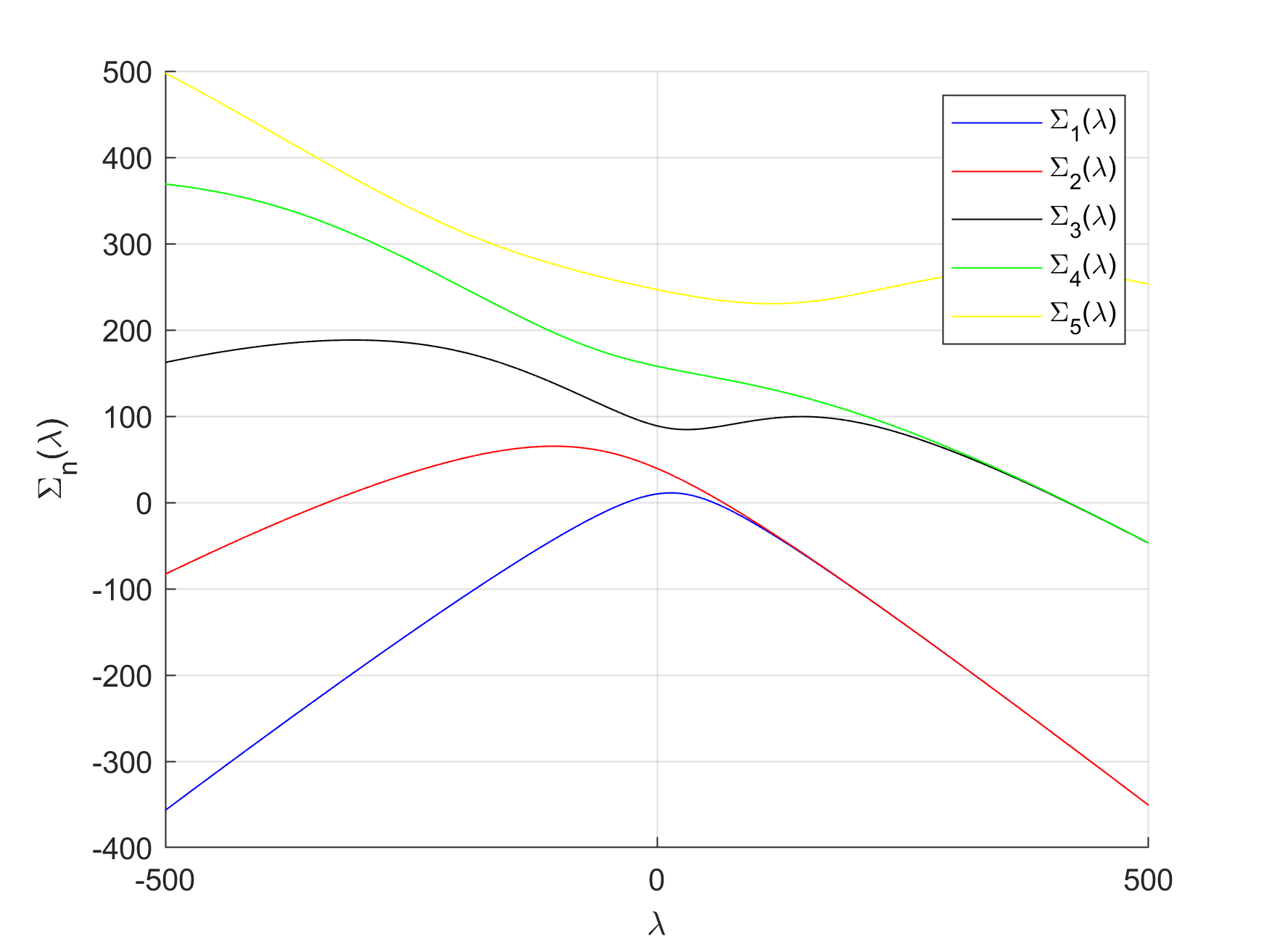}\hspace{-0.5cm} \includegraphics[width=0.51\linewidth]{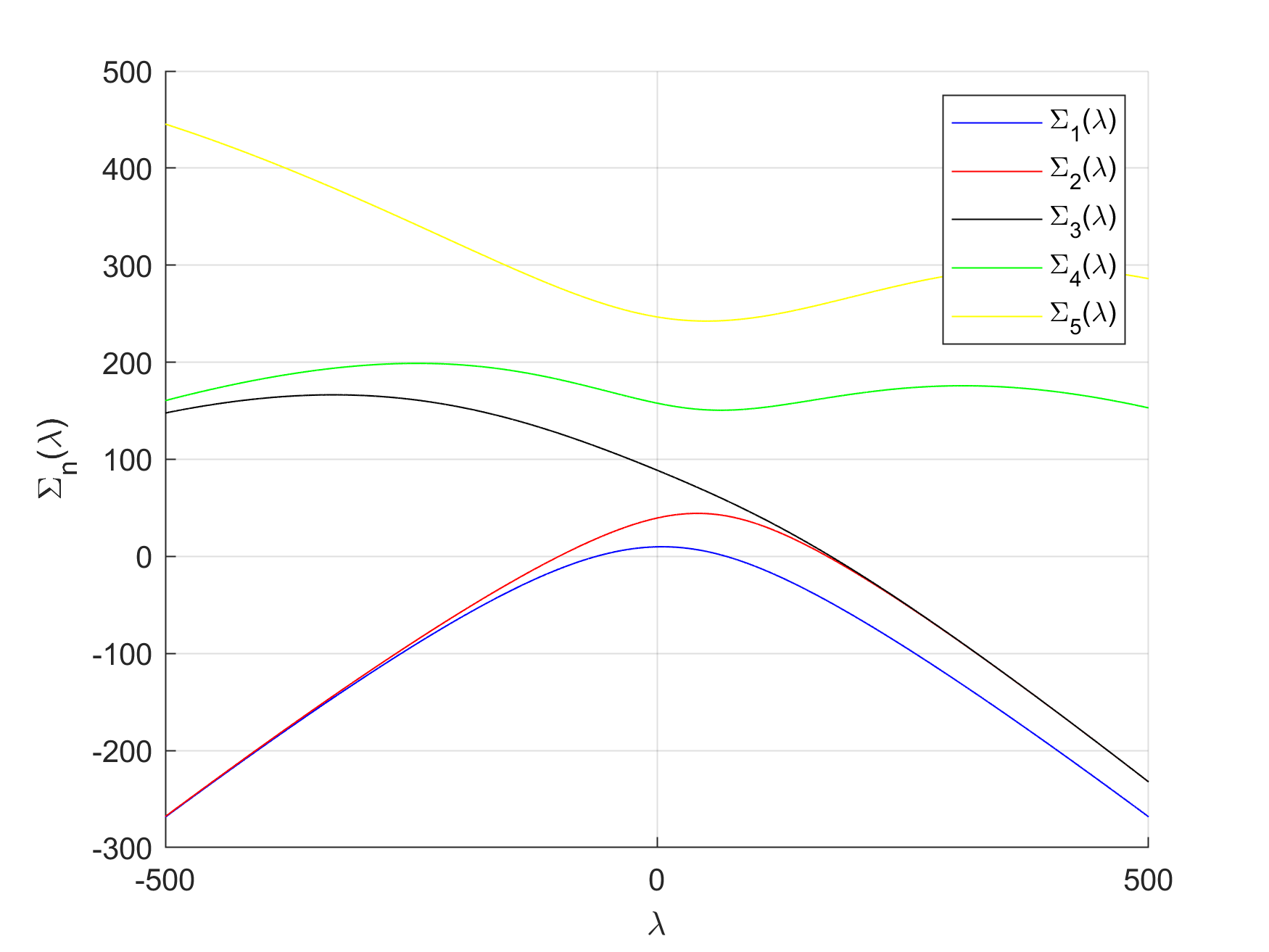}
	\caption{The curves $\Sigma_n(\lambda)$ for $1\leq n\leq 5$ with $m(x)=\sin(n\pi x)$, $n=3, 5$.}
	\label{Fig3}
\end{figure}
The next result establishes that, as already suggested by Figures \ref{Fig1} and \ref{Fig2}, the nodal
eigencurves, $\Sigma_n(\l)$, cannot be  concave for the choice \eqref{1.8} if $n\geq k+1$.  We conjecture that, in general, for that particular choice, $\Sigma_n$ is concave if $n\leq k$. Therefore,
$\Sigma_n$ should be concave if, and only if, $n\leq k$. But the analysis of the concavity
when $n\leq k$ for the choice \eqref{1.8} remains outside the general scope of this paper.

\begin{theorem}
\label{th2.1}
Assume \eqref{1.8} for some integer $k\geq 1$. Then, as soon as $n\geq k+1$,
\begin{equation}
\label{2.5}
   \ddot{ \Sigma}_n(0)>0\quad \hbox{for all} \;\; n\geq k+1.
\end{equation}
Therefore, by \eqref{2.4}, $\l=0$ is a local minimum of $\Sigma_n(\l)$ and, in particular,
$\Sigma_n(\l)$ cannot be concave.
\end{theorem}
\begin{proof}
Since $\Sigma_n(\l)$ is algebraically simple for all $n\geq 1$, we already know that $\Sigma_n(\l)$ is analytic, by some well known perturbation results of Kato \cite{KaA}.
Moreover, the eigenfunction associated to $\Sigma_n(\l)$, denoted by $\v_{[n,\l]}$, can be chosen to be analytic in $\l$ by normalizing it so that
\begin{equation}
\label{2.6}
  \int_0^1 \v_{[n,\l]}^2(x)\,dx =\frac{1}{2}.
\end{equation}
By definition, $\v_{[n,\l]}(0)=\v_{[n,\l]}(1)=0$ and
\begin{equation}
\label{2.7}
   -\v_{[n,\l]}''(x)=\l m(x) \v_{[n,\l]}(x)+\Sigma_n(\l)\v_{[n,\l]}(x) \quad
   \hbox{for all}\;\; x\in (0,1).
\end{equation}
Thus, since $\Sigma_n(0)=(n\pi)^2$, particularizing \eqref{2.7} at $\l=0$ and taking into account \eqref{2.6}, it becomes apparent that  actually $\v_{[n,\l]}$ is an analytic perturbation of the eigenfunction
\begin{equation*}
  \v_{[n,0]}(x)=\sin(n\pi x),\qquad x\in [0,1].
\end{equation*}
Moreover, differentiating \eqref{2.7} with respect to $\l$ yields
\begin{equation}
\label{2.8}
    -\dot{\v}_{[n,\l]}''(x)=\l m \dot{\v}_{[n,\l]}+m\v_{[n,\l]} +\dot{\Sigma}_n(\l)\v_{[n,\l]}
    +\Sigma_n(\l) \dot{\v}_{[n,\l]}\quad \hbox{in}\;\; (0,1).
\end{equation}
Thus, since $\Sigma_n(0)=(n\pi)^2$ and $\dot{\Sigma}_n(0)=0$, particularizing \eqref{2.8} at $\l=0$
shows that $\dot{\v}_{[n,0]}$ solves the problem
\begin{equation}
\label{2.9}
   \left\{ \begin{array}{l} [-D^2-(n\pi)^2]u=m\v_{[n,0]}  \quad \hbox{in}\;\; (0,1),\\[1ex]
   u(0)=u(1)=0. \end{array}  \right.
\end{equation}
In order to find out $\dot{\v}_{[n,0]}$, we first determine the general solution of the linear
inhomogeneous equation
\begin{equation}
\label{2.10}
    [-D^2-(n\pi)^2]u=m(x)\sin (n\pi x).
\end{equation}
To get it, we will set $v:=u'$ in order to vary coefficients in the first order system associated to \eqref{2.10},\begin{equation}
\label{2.11}
  \left( \begin{array}{c} u' \\[1ex] v' \end{array} \right) =
   \left( \begin{array}{cc} 0 & 1 \\[1ex] -(n\pi)^2 & 0 \end{array} \right)
   \left( \begin{array}{c} u \\[1ex] v \end{array} \right)+
    \left( \begin{array}{c} 0 \\[1ex] -m(x)\sin(n\pi x)\end{array} \right).
\end{equation}
Since
$$
  W(x):= \left( \begin{array}{cc} \cos(n\pi x) & \sin (n\pi x) \\[1ex]
  -n\pi \sin(n\pi x) & n\pi \cos (n\pi x) \end{array}\right)
$$
is a fundamental matrix of solutions for the homogeneous linear system associated to \eqref{2.11},
the change of variable
$$
  \left( \begin{array}{c} u \\[1ex] v \end{array} \right) = W(x) \left( \begin{array}{c} c_1(x) \\[1ex] c_2(x) \end{array} \right)
$$
transforms \eqref{2.11} into the equivalent system
$$
  W(x) \left( \begin{array}{c} c_1'(x) \\[1ex] c_2'(x) \end{array} \right) =
  \left( \begin{array}{c} 0 \\[1ex] -m(x)\sin(n\pi x)\end{array} \right),
$$
whose solution, according to Cramer rule, is given through
\begin{equation*}
  c_1'(x) = \frac{1}{n\pi}m(x)\sin^2(n\pi x),\qquad
  c_2'(x)=  \frac{-1}{n\pi}m(x)\sin(n\pi x)\cos(n\pi x).
\end{equation*}
Thus,
\begin{align*}
  c_1(x) & =\frac{1}{n\pi} \int_0^x m(s)\sin^2(n\pi s)\,ds + A, \\[1ex]
  c_2(x) & =  -\frac{1}{n\pi}\int_0^x m(s)\sin(n\pi s)\cos(n\pi s)\,ds +B,
\end{align*}
for some constants $A, B \in\R$. Therefore, the general solution of \eqref{2.10} is given by
\begin{align*}
  u(x) & = \cos(n\pi x)c_1(x)+\sin(n\pi x)c_2(x) \\[1ex] & =
  \cos(n\pi x) \left( A+\frac{1}{n\pi} \int_0^x m(s)\sin^2(n\pi s)\,ds\right) \\[1ex]
  & \;\;\;   +\sin(n\pi x) \left( B - \frac{1}{n\pi}\int_0^x m(s)\sin(n\pi s)\cos(n\pi s)\,ds\right) \\[1ex] & = A \cos(n\pi x)+ B \sin (n\pi x) + p(x),
\end{align*}
where
\begin{equation}
\label{2.12}
  p(x):=  \frac{1}{n\pi}\int_0^x m(s) \sin(n\pi s) \sin[n\pi(s-x)]\,ds,\qquad x\in[0,1],
\end{equation}
is a particular solution of \eqref{2.10}. It is the solution obtained by making the choice $A=B=0$.  Obviously, $p(0)=0$. Moreover, by \eqref{1.8},
\begin{align*}
  p(1) & = \int_0^1 m(s) \sin(n\pi s)  \sin[n\pi(s-1)] \,ds \\[1ex]
  & = (-1)^{n} \int_0^1 \sin(2k\pi s) \sin^2(n\pi s)\,ds =0,
\end{align*}
because the integrand,
$$
  \t(s):=\sin(2k\pi s) \sin^2(n\pi s),\qquad s\in [0,1],
$$
satisfies $\t(1-s)=-\t(s)$ for all $s\in [0,1]$ and hence, it is odd about $0.5$.
As we are interested in solving  \eqref{2.9},  we should make the choice
$$
  0 = u(0)= A + p(0)=A.
$$
Thus,
\begin{equation*}
  \dot{\v}_{[n,0]}(x)= B\sin(n\pi x)+p(x), \qquad x\in [0,1],
\end{equation*}
for some constant $B\in\R$. To determine $B$, we can proceed as follows. Differentiating
\eqref{2.6} with respect to $\l$ and particularizing the resulting identity at $\l=0$ yields
$$
  0= \int_0^1 \v_{[n,0]}(x)\dot{\v}_{[n,0]}(x)\,dx =B\int_0^1 \sin^2(n\pi x)\,dx +
  \int_0^1 \sin(n\pi x)p(x)\,dx.
$$
Consequently,
$$
  B=-2\int_0^1 \sin(n\pi x)p(x)\,dx
$$
and therefore,
\begin{equation}
\label{2.13}
  \dot{\v}_{[n,0]}(x)= -2\left( \int_0^1 \sin(n\pi s)p(s)\,ds \right)
  \sin(n\pi x)+p(x), \qquad x\in [0,1].
\end{equation}
To find out $\ddot{\Sigma}_n(0)$, we can differentiate with respect to $\l$ the identity \eqref{2.8}.
After rearranging terms, this provides us with the identity
\begin{equation*}
  [-D^2-\l m -\Sigma_n(\l)]\ddot{\v}_{[n,\l]}= 2 m \dot{\v}_{[n,\l]}+2 \dot{\Sigma}_n(\l) \dot{\v}_{[n,\l]}+
  \ddot{\Sigma}_n(\l)\v_{[n,\l]}.
\end{equation*}
Thus, particularizing at $\l=0$ yields
\begin{equation}
\label{2.14}
  [-D^2 -(n\pi)^2]\ddot{\v}_{[n,0]}=2 m \dot{\v}_{[n,0]}+\ddot{\Sigma}_n(0)\v_{[n,0]}
\end{equation}
and hence, multiplying \eqref{2.14} by $\v_{[n,0]}$ and integrating in $(0,1)$ it is apparent that
\begin{equation}
\label{2.15}
  \ddot{\Sigma}_n(0)=-4\int_0^1 m(x) \dot{\v}_{[n,0]}(x) \v_{[n,0]}(x)\,dx.
\end{equation}
Therefore, substituting \eqref{2.13} into \eqref{2.15} and using \eqref{1.8} yields
\begin{align*}
  \ddot{\Sigma}_n(0) & =-4\int_0^1 m(x)\v_{[n,0]}(x)p(x)\,dx\\ & =
  -4\int_0^1 \sin(2k\pi x)\sin(n\pi x)\left[\frac{1}{n\pi}\int_0^x \sin(2k\pi s)\sin(n\pi s)\sin(n\pi(s-x))\,ds\right]\,dx.
\end{align*}
Finally, we need the trigonometric formulas
\begin{eqnarray}
\sin x\sin y  &=& \tfrac{1}{2}\left[\cos(x-y)-\cos(x+y)\right],\label{2.16}\\
\sin x\cos y &=& \tfrac{1}{2}\left[\sin(x-y)+\sin(x+y)\right],\label{2.17}
\end{eqnarray}
to simplify the integrands arising in integrals of  $\ddot{\Sigma}_n(0)$. First, we will ascertain the function $p(x)$. For this, we use the formula \eqref{2.16} on $\sin(2k\pi s)\sin(n\pi s)$ and then the formula \eqref{2.17} to simplify the integrand in $p(x)$. Then, integrating yields
\begin{equation}
p(x) = -\frac{1}{8 \pi^2}\left[ \frac{\cos(\pi x(2k-n))}{k(n-k)} + \frac{\cos(\pi x(2k+n))}{k(n+k)} - \frac{n \cos(n\pi x)}{k(n^2 - k^2)} \right].\label{2.18}
\end{equation}
After substituting \eqref{2.18}  into the formula for $\ddot{\Sigma}_n(0)$, we can again use the formulas  \eqref{2.16} and \eqref{2.17} to simplify the underlying integrands, which can then be directly integrated. The result can be simplified to get the final formula
\begin{equation*}
	\ddot{\Sigma}_n(0) = \frac{1}{4\pi^2 (n^2 - k^2)}.
\end{equation*}
Obviously,  $n^2 - k^2 > 0$ if $n\geq k+1$, and therefore $\ddot{\Sigma}_n(0)>0$. Hence, the eigencurves $\Sigma_n(\lambda)$ for $n\geq k+1$ are convex in a neighborhood of $\lambda = 0$ and thus they cannot be globally  concave.
\end{proof}

\section{Global bifurcation of nodal solutions}

\noindent Since $\Sigma_1(0)=\pi^2$, for every $\mu < \pi^2$ the set $\Sigma_1^{-1}(\mu)$
consists of two points,
$$
  \l_-(\mu)<0<\l_+(\mu),
$$
such that
$$
  \lim_{\mu\ua \pi^2}\l_\pm(\mu)=0.
$$
Moreover, owing to Theorem 9.4 of \cite{LG13},
$$
  \dot{\Sigma}_1(\l_-(\mu))>0 \quad \hbox{and}\quad \dot{\Sigma}_1(\l_+(\mu))<0.
$$
Thus, by the main theorem of Crandall and Rabinowitz \cite{CR} (one can see also Chapter 2 of \cite{LG01}),
$\l= \l_\pm(\mu)$ are the unique bifurcation values of $\l$ to positive solutions of \eqref{1.1} from $u=0$. The first plot of Figure 1 of L\'{o}pez-G\'{o}mez and Molina-Meyer \cite{LGMM} shows one of those  bifurcation diagrams for the special choice \eqref{1.2} of $a(x)$ with
\begin{equation}
\label{3.1}
  m(x)=\sin(2\pi x),\qquad x\in [0,1].
\end{equation}
Trying to complement the numerical experiments of \cite{LGMM} with our new findings here, all the numerical experiments of this section has been carried out for this special choice of $m(x)$.
As  $\mu$ grows up to reach the critical value $\mu=\pi^2$, the set of positive solutions of \eqref{1.1}
bifurcating from $u=0$ consists of one single closed loop bifurcating from $u=0$ at the single point $\l=0$. These loops, separated away from $u=0$, are persistent for a large range of values of $\mu>\pi^2$, until they shrink to a single point before disappearing at some critical value of the parameter $\mu$
(see \cite[Fig. 1]{LGMM}).
\par
According to Theorem \ref{th2.1}, $\Sigma_2(\l)$ is not concave if \eqref{3.1} holds, which is
clearly illustrated by simply looking at the plot of $\Sigma_2(\l)$ superimposed in Figure 1.
This feature has important implications concerning the structure of the set
of 1-node solutions of \eqref{1.1}. Indeed, according to
the plot of $\Sigma_2(\l)$, for every $\mu < (2\pi)^2$, the set $\Sigma_2^{-1}(\mu)$ consists of
two single values $\l_-(\mu)<0<\l_+(\mu)$ with $\dot{\Sigma}_2(\l_-(\mu))>0$ and
$\dot{\Sigma}_2(\l_+(\mu))<0$. Thus, according to \cite[Th. 9.4]{LG13}, the transversality condition of Crandall and Rabinowitz \cite{CR} holds at $(\l,u)=(\l_\pm(\mu),0)$. Thus, an analytic
curve of 1-node solutions of \eqref{1.1} emanates from $u=0$
at each of these values of $\l$, $\l_\pm(\mu)$. Figure \ref{Fig4a} shows
the plots of these two curves for the value of the parameter $\mu=35$. Our numerical experiments suggest that they are separated away from each other. In this bifurcation diagram, as well as in all the remaining ones, we are representing the values of the parameter $\l$, in abscisas, versus the $L^2$-norm of the computed solutions, in ordinates. So, each point on the curves of
the bifurcation diagrams, $(\l,u)$,  represent a value of $\l$ and a nodal solution $u$ of \eqref{1.1} for that particular value of $\l$.
\par
\begin{figure}[h!]
\begin{subfigure}[t]{0.45\textwidth}
	\centering
	\includegraphics[scale=0.53]{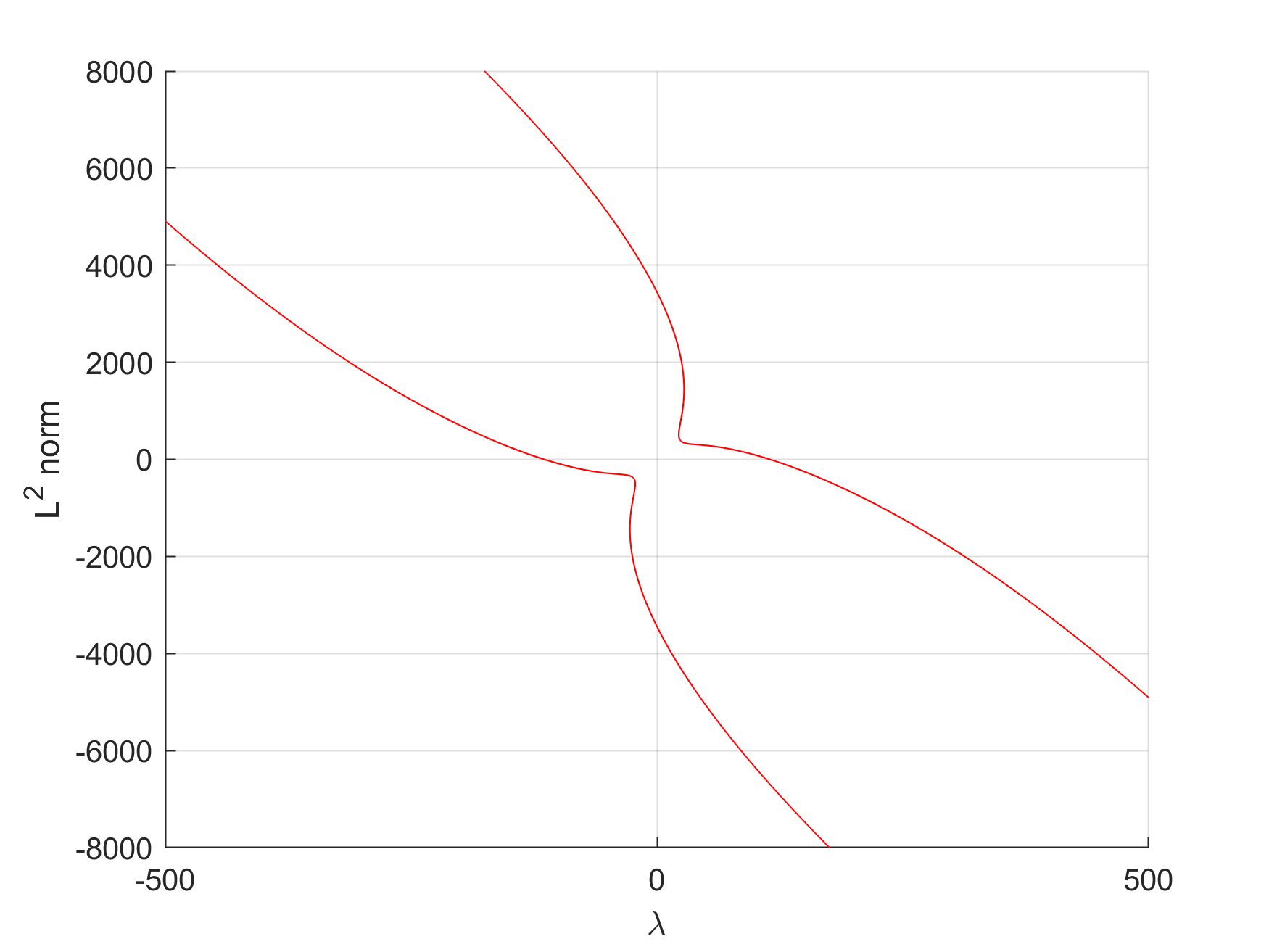}
	\caption{$\mu = 35$}\label{Fig4a}
\end{subfigure}%
\begin{subfigure}[t]{0.45\textwidth}
	\centering
	\includegraphics[scale=0.53]{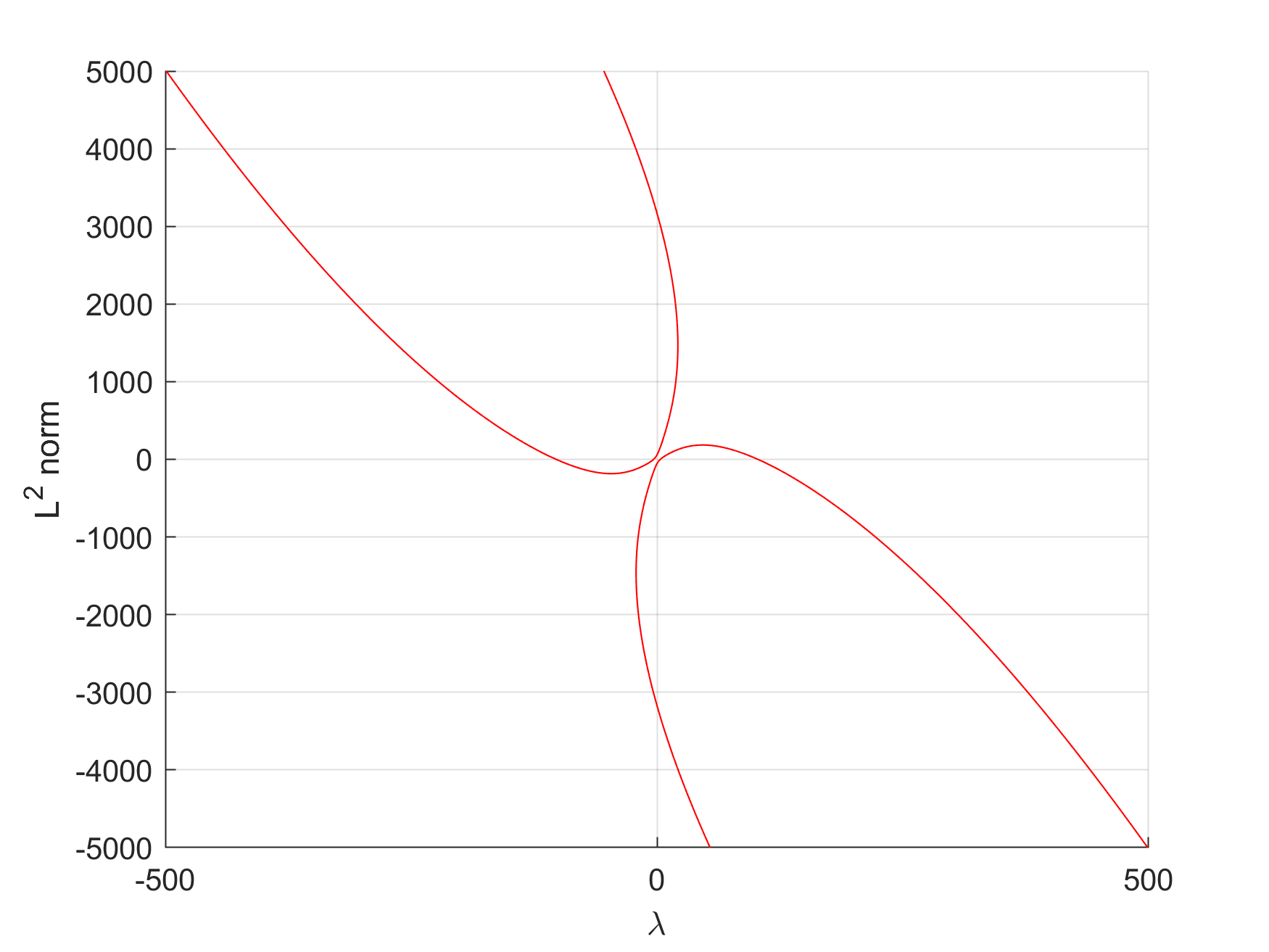}
	\caption{$\mu = 39.6$}\label{Fig4b}
\end{subfigure}

\begin{subfigure}[t]{0.45\textwidth}
	\centering
	\includegraphics[scale=0.53]{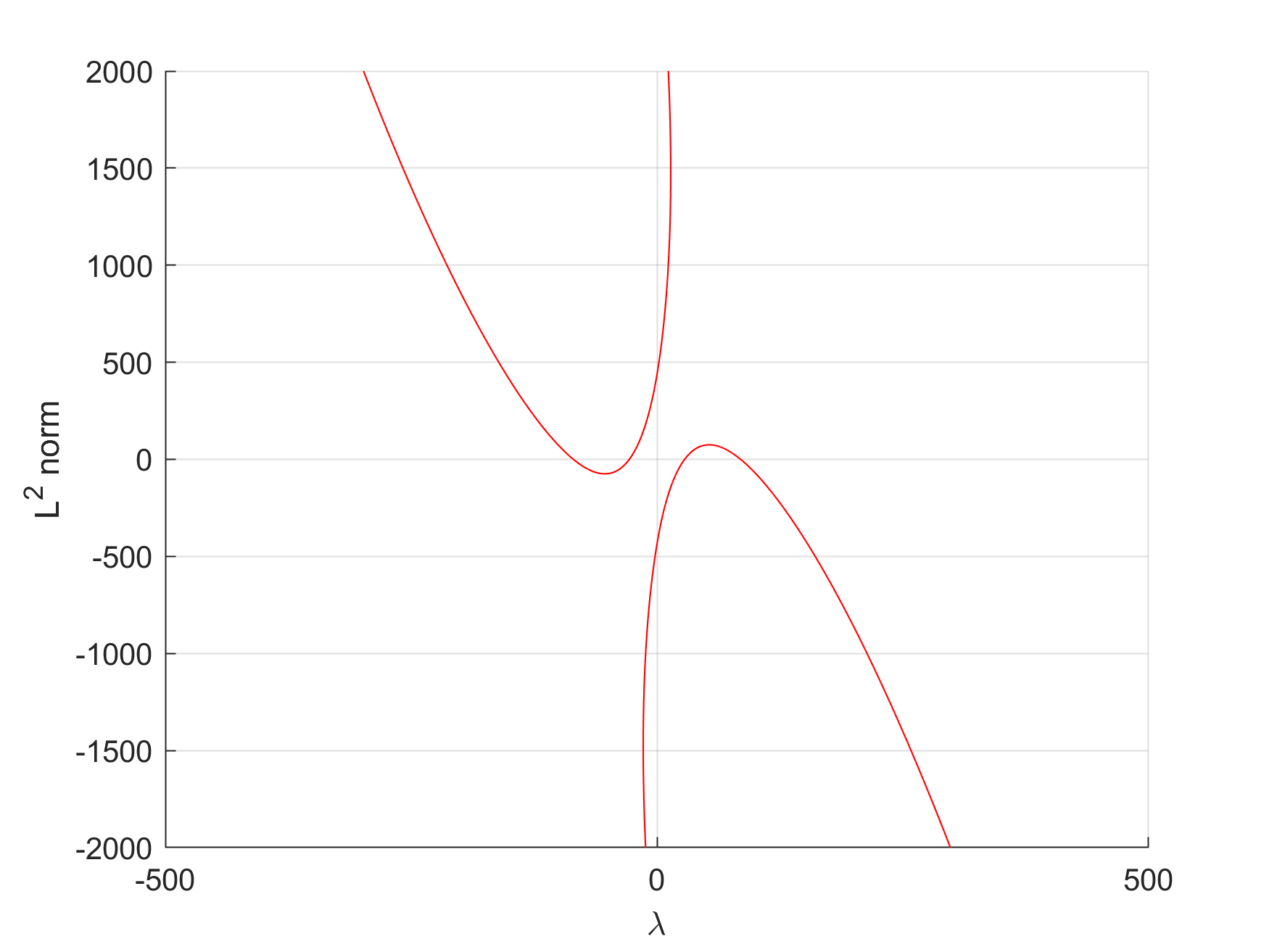}
	\caption{$\mu = 45$}\label{Fig4c}
\end{subfigure}%
\begin{subfigure}[t]{0.45\textwidth}
	\centering
	\includegraphics[scale=0.53]{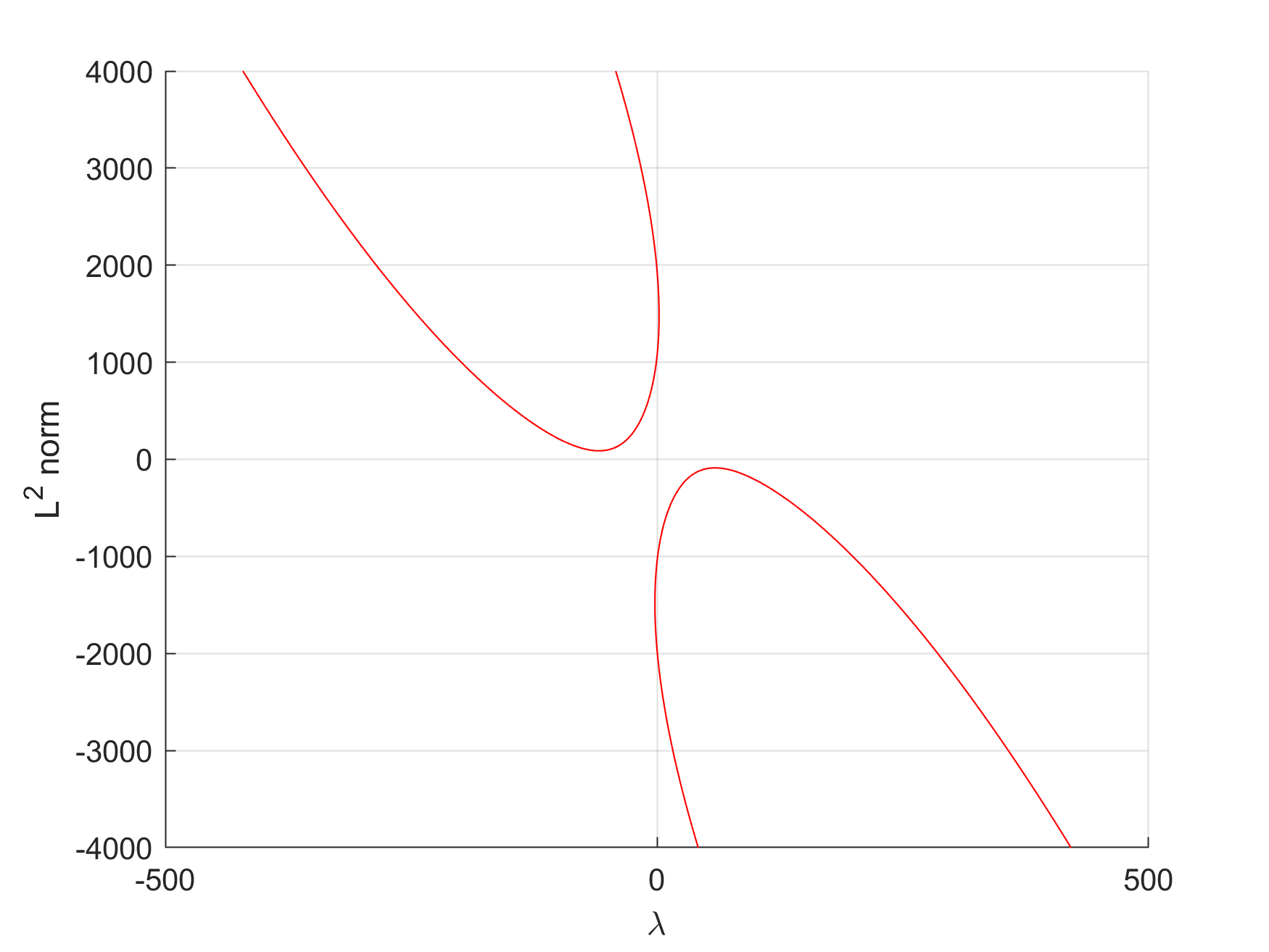}
	\caption{$\mu = 54$}\label{Fig4d}
\end{subfigure}
\caption{Four representative bifurcation diagrams of 1-node solutions.}
\label{Fig4}
\end{figure}
When $\mu$ grows up to reach the critical value $(2\pi)^2$, the two previous components become
closer and closer until they meet at $\l=0$ at $\mu=(2\pi)^2$, where the set of bifurcation points to 1-node solutions from $u=0$ consists of the points $(\l_\pm((2\pi)^2),0)$ plus $(0,0)$. This is the situation sketched by Figure \ref{Fig4b}, where we have plotted the global bifurcation diagram computed for
$$
  \mu=39.6>39.4786 \sim (2\pi)^2.
$$
When $\mu \in ((2\pi)^2,\mu_2)$, where $\mu_2$ is given by \eqref{1.7}, the set $\Sigma_2^{-1}(\mu)$ consists of four values: two negative, $\l_{-[1,2]}(\mu)<\l_{-[2,2]}(\mu)<0$,  plus two positive,  $0<\l_{+[2,2]}(\mu)<\l_{+[1,2]}(\mu)$. Moreover, by Proposition \ref{pr2.2}, it is apparent that
$$
  0< \l_{+[2,2]}(\mu)=-\l_{-[2,2]}(\mu)<\l_{+[1,2]}(\mu)=-\l_{-[1,2]}(\mu).
$$
Furthermore, as suggested by our numerical experiments,
$$
  \dot{\Sigma}_2(\l_{-[1,2]}(\mu))>0, \quad \dot{\Sigma}_2(\l_{-[2,2]}(\mu))<0, \quad
  \dot{\Sigma}_2(\l_{+[2,2]}(\mu))>0, \quad \dot{\Sigma}_2(\l_{+[1,2]}(\mu))<0.
$$
Thus, again the transversality condition of \cite{CR} holds at each of these critical values of
the parameter $\l$. Therefore, \eqref{1.1} should possess four analytic curves filled in by  1-node
solutions bifurcating from $u=0$ at each of these critical values of the parameter $\l$.
Figure \ref{Fig4c} shows the global bifurcation diagram of 1-node solutions bifurcating from these
four bifurcation points  that we have computed for the choice $\mu =45$. Once again, the set of 1-node solutions consists of two components.
\par
Actually, as soon as the transversality condition of Crandall and Rabinowitz \cite{CR}
holds, the generalized algebraic multiplicity  of Esquinas and L\'{o}pez-G\'{o}mez \cite{ELG,LG01}, $\chi$, equals 1 and hence, thanks to Theorem 5.6.2 of L\'{o}pez-G\'{o}mez \cite{LG01}, the Leray--Schauder
index of $u=0$, as a solution of \eqref{1.1}, changes as $\l$ crosses each of these values.
Therefore, each of the components of the set of non-trivial solutions of \eqref{1.1}
emanating from $u=0$ at each of these critical values of the primary parameter $\l$ satisfies
the global alternative of Rabinowitz \cite{Ra1}, i.e., either it is unbounded in $\R \times \mc{C}[0,1]$, or it meets the trivial solution in, at least, two of these singular values.
\par
Each of the two components plotted in Figure \ref{Fig4c} bifurcates from two different points of $(\l,0)$ and, according to our numerical experiments, both seem to be unbounded.
The problem of ascertaining their precise global behavior remains open in this paper.
As $\mu$ increases and crosses the critical value $\mu_2$, these two components abandone the trivial curve and  stay separated away from the trivial solution. So, they became isolas. Figure \ref{Fig4d} shows the plots  of these components for the  choice $\mu =54$. In Figure \ref{Fig5} we have plotted
some distinguished solutions with 1-node along some of the pieces of the global bifurcation diagrams already plotted in Figure \ref{Fig4}. Precisely, Figure \ref{Fig5b} shows a series of solutions with one node along the bifurcation diagram plotted on Figure \ref{Fig5a}, which is a magnification of a piece of the left component of Figure \ref{Fig4a}, and  Figure \ref{Fig5d} shows a series of solutions with one node along the bifurcation diagram plotted in Figure \ref{Fig5c}, which is a
magnification of a piece of the left component plotted in Figure \ref{Fig4d}. The colors of each of these 1-node solutions corresponds with the color of the piece of the bifurcation diagram on the left where they
are coming from.
\begin{figure}[h!]
\begin{subfigure}[t]{0.45\textwidth}
	\centering
	\includegraphics[scale=0.53]{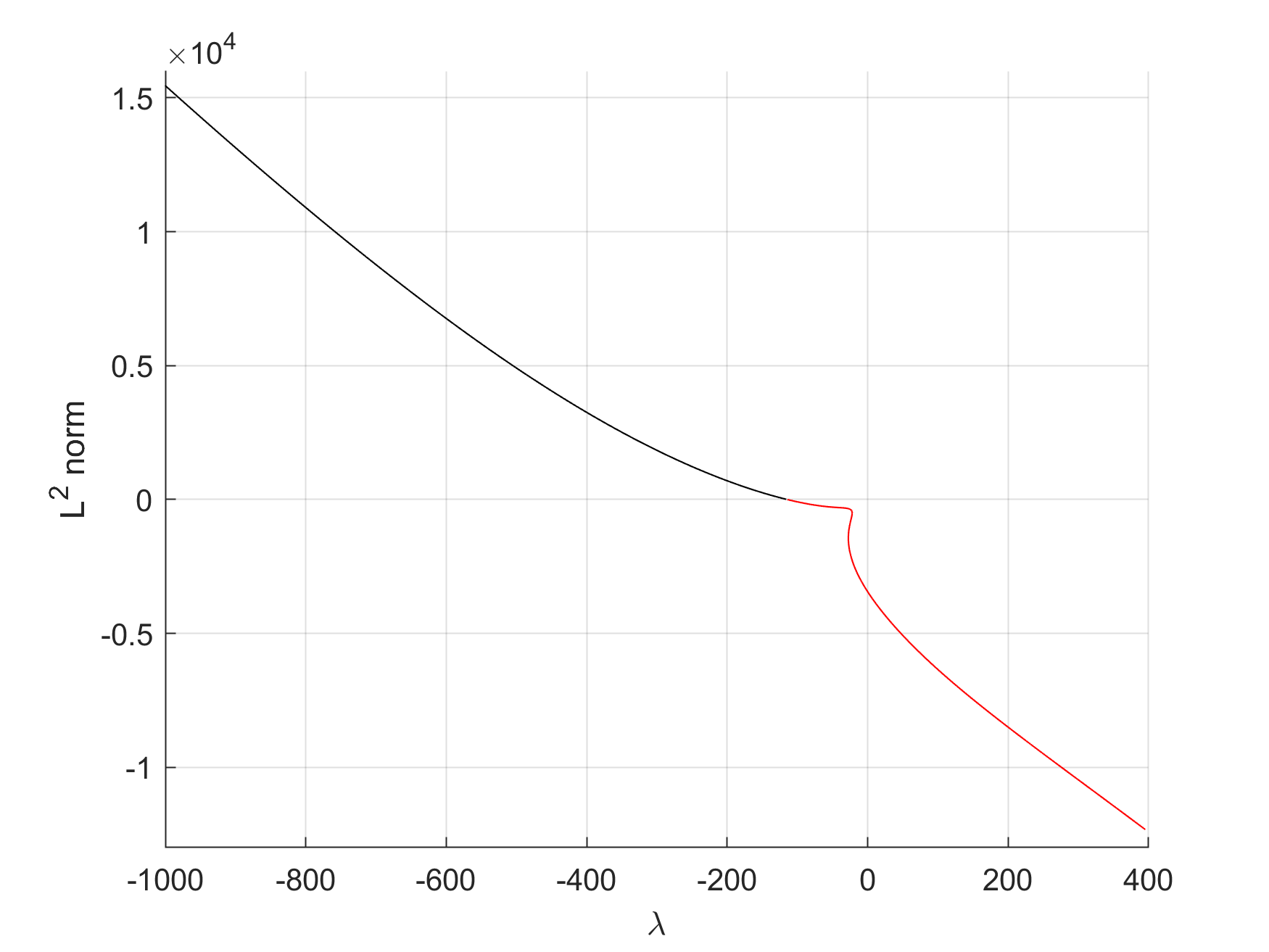}
	\caption{Left branch of Figure \ref{Fig4a} ($\mu = 35$)}\label{Fig5a}
\end{subfigure}%
\begin{subfigure}[t]{0.45\textwidth}
	\centering
	\includegraphics[scale=0.53]{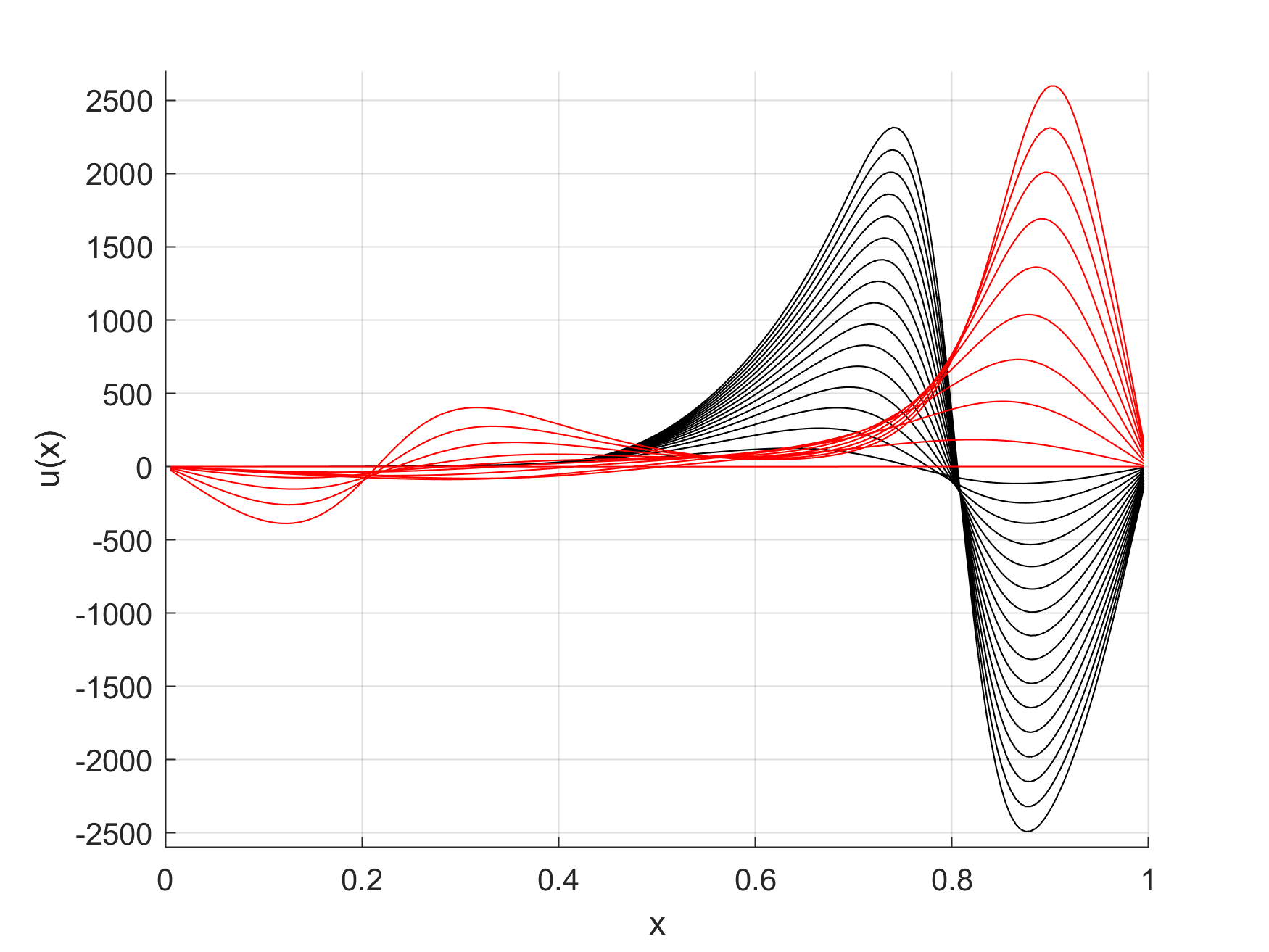}
	\caption{Plots of some solutions on the left}\label{Fig5b}
\end{subfigure}

\begin{subfigure}[t]{0.45\textwidth}
	\centering
	\includegraphics[scale=0.53]{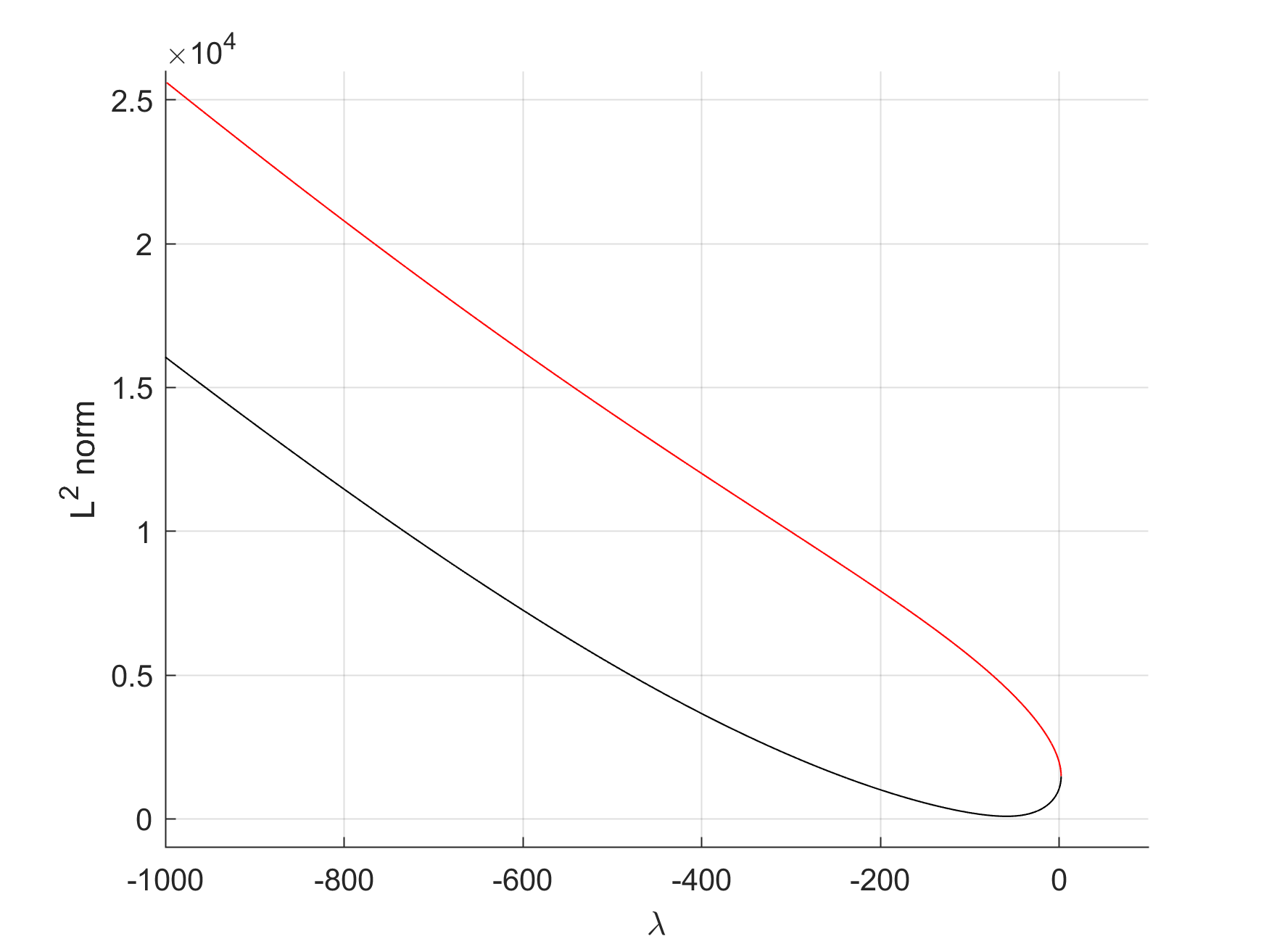}
	\caption{Left branch of  Figure \ref{Fig4d} ($\mu = 54$)}\label{Fig5c}
\end{subfigure}%
\begin{subfigure}[t]{0.45\textwidth}
	\centering
	\includegraphics[scale=0.53]{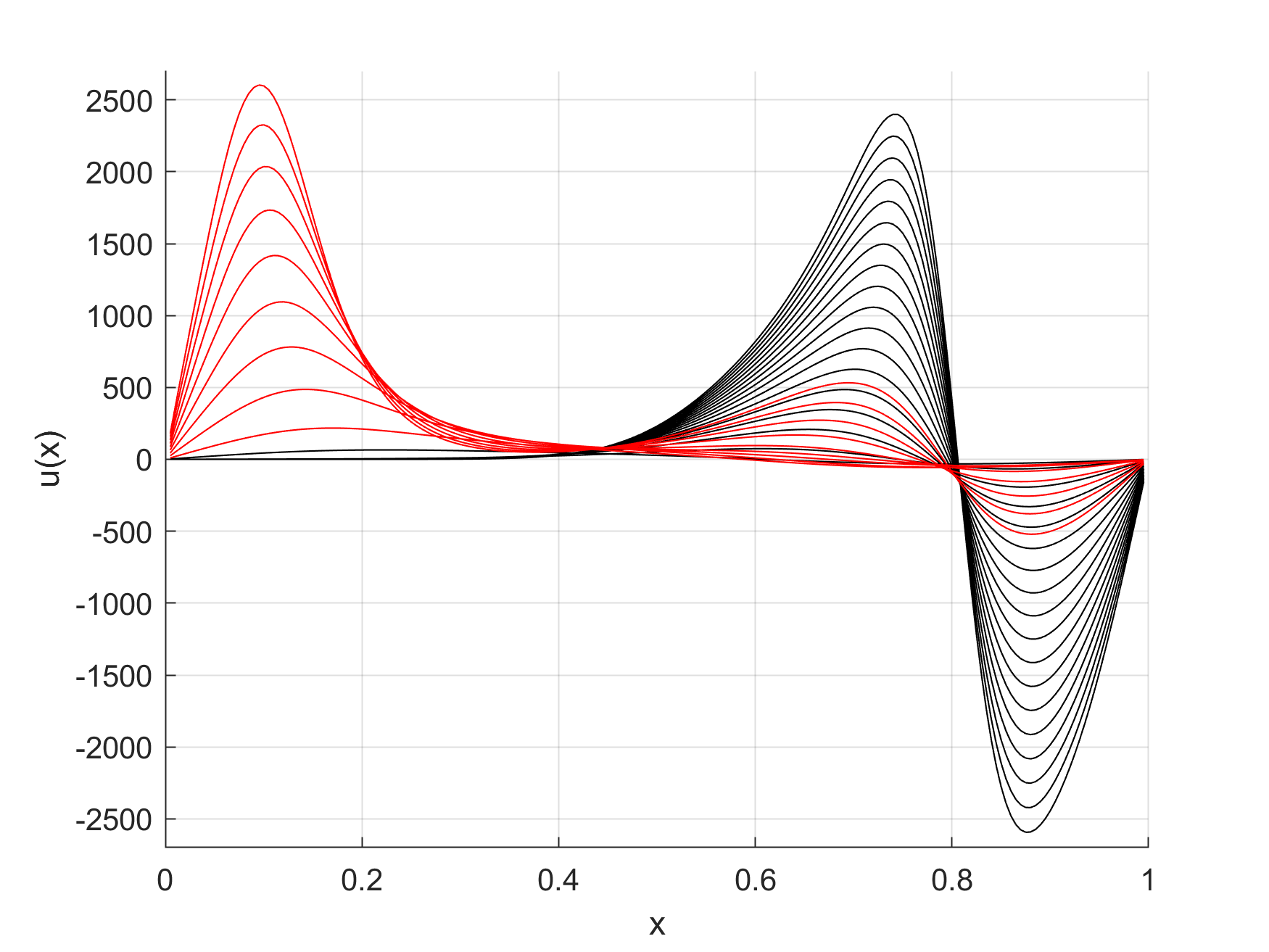}
	\caption{Plots of some solutions on the left}\label{Fig5d}
\end{subfigure}
\caption{A series of plots of 1-node solutions (right)  along some of the components of Figure \ref{Fig4} (left).}
\label{Fig5}
\end{figure}
\par
Similarly, according to Theorem \ref{th2.1}, for the special choice \eqref{3.1}, the third eigencurve,  $\Sigma_3(\l)$, is far from concave if \eqref{3.1} holds. This becomes apparent by simply having a look
at the plot of $\Sigma_3(\l)$ superimposed in Figure \ref{Fig1}.  According to it, for every $\mu\in ((3\pi)^2,\mu_3)$, the set $\Sigma_3^{-1}(\mu)$ consists of two negative eigenvalues,
$\l_{-[1,3]}(\mu)<\l_{-[2,3]}(\mu)<0$,  plus two positive eigenvalues,  $0<\l_{+[2,3]}(\mu)<\l_{+[1,3]}(\mu)$. Moreover, by Proposition \ref{pr2.2},
$$
  0< \l_{+[2,3]}(\mu)=-\l_{-[2,3]}(\mu)<\l_{+[1,3]}(\mu)=-\l_{-[1,3]}(\mu)
$$
and, according to our numerical experiments,
$$
  \dot{\Sigma}_2(\l_{-[1,3]}(\mu))>0, \quad \dot{\Sigma}_2(\l_{-[2,3]}(\mu))<0, \quad
  \dot{\Sigma}_2(\l_{+[2,3]}(\mu))>0, \quad \dot{\Sigma}_2(\l_{+[1,3]}(\mu))<0.
$$
Thus, the transversality condition of \cite{CR} holds at each of these critical values. Therefore, owing to the local bifurcation theorem of \cite{CR}, an analytic curve of 2-node solutions emanates from $u=0$
at each of these four singular values of $\l$.
The first three plots of Figure \ref{Fig6} show these curves
for three different values of the secondary parameter $\mu$. Namely: $\mu=105$, $\mu =108.1$ and $\mu=110$, respectively. All these  values  of $\mu$ are bellow $\mu_3$. The last plot of Figure \ref{Fig6} has been computed for $\mu=140 >\mu_3$ and shows three components of 2-node solutions separated away from $u=0$. For this value of $\mu$ no solution with 2 interior nodes can bifurcate from $u=0$.
\par
\begin{figure}[h!]
	\begin{subfigure}[t]{0.45\textwidth}
		\centering
		\includegraphics[scale=0.53]{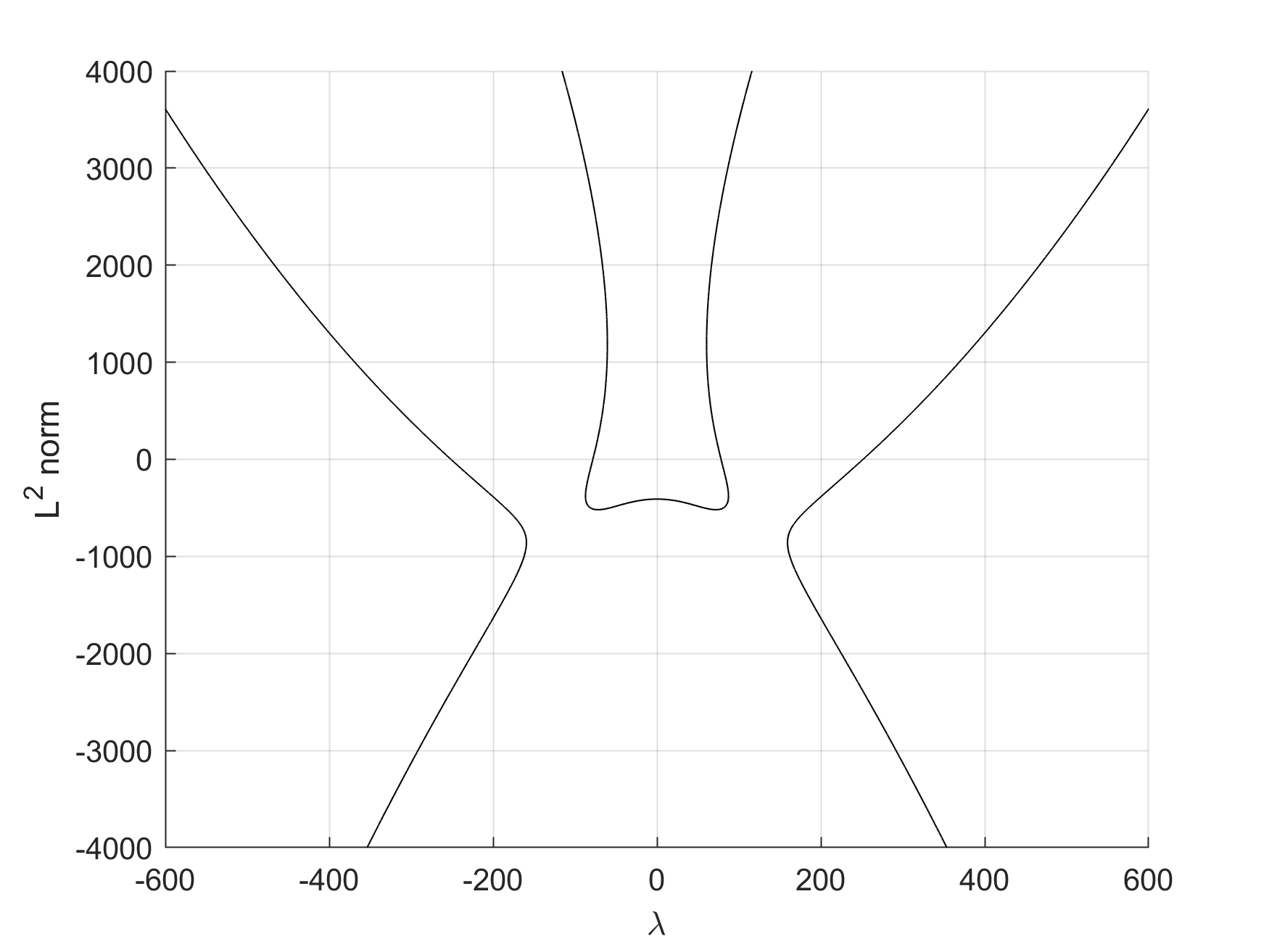}
		\caption{$\mu = 105$}\label{Fig6a}
	\end{subfigure}%
	\begin{subfigure}[t]{0.45\textwidth}
		\centering
		\includegraphics[scale=0.53]{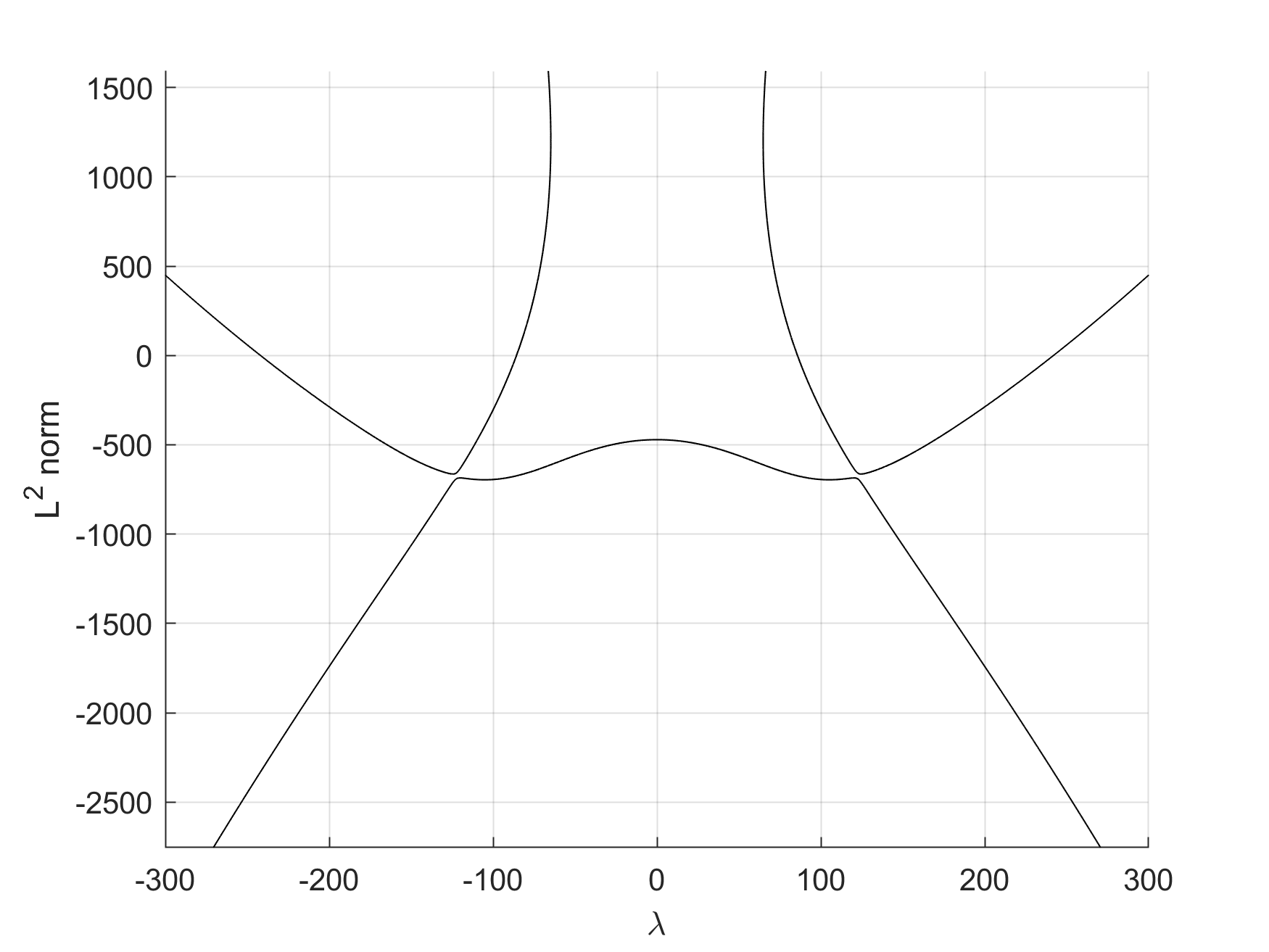}
		\caption{$\mu = 108.1$}\label{Fig6b}
	\end{subfigure}
	
	\begin{subfigure}[t]{0.45\textwidth}
		\centering
		\includegraphics[scale=0.53]{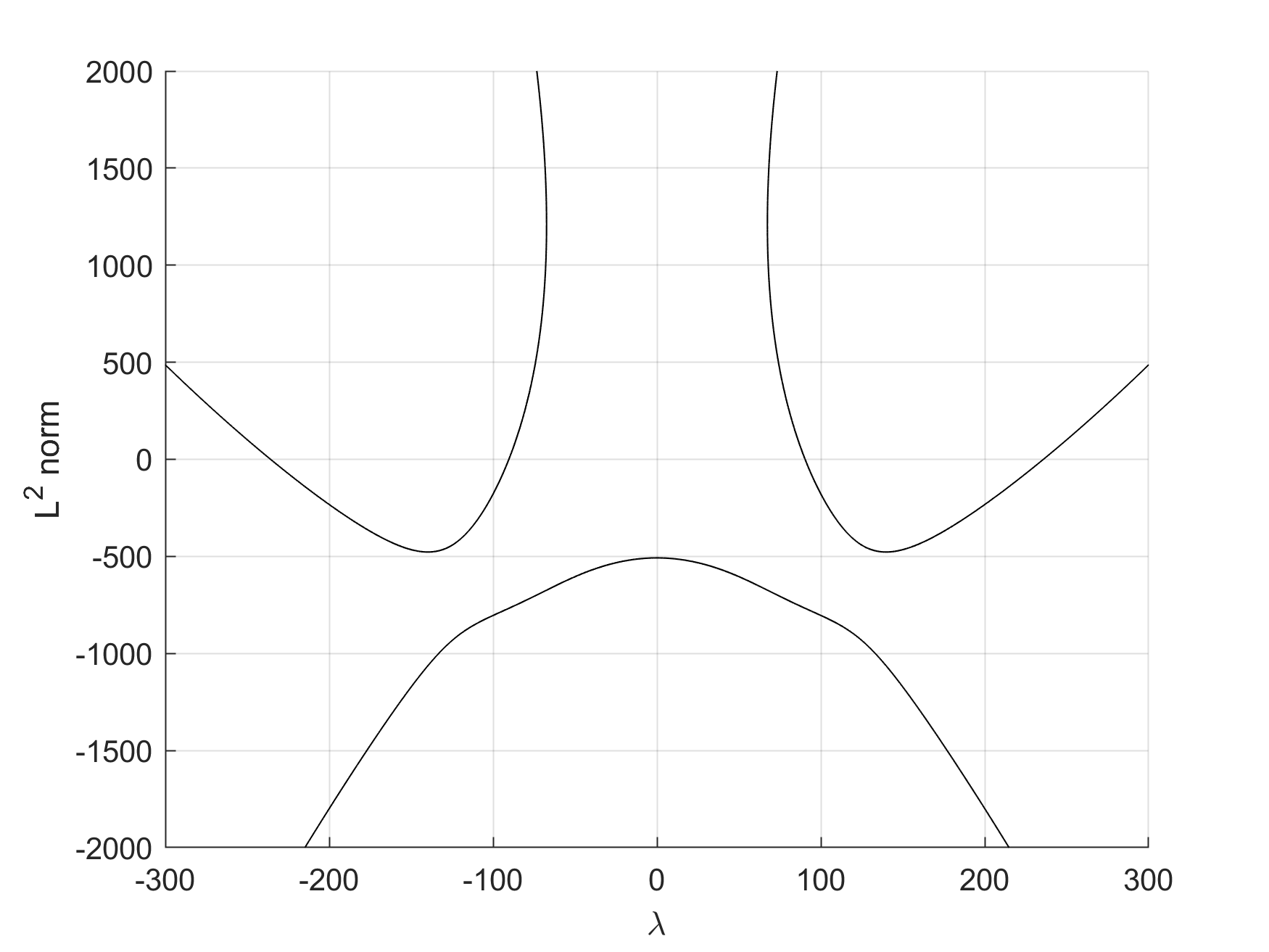}
		\caption{$\mu = 110$}\label{Fig6c}
	\end{subfigure}%
	\begin{subfigure}[t]{0.45\textwidth}
		\centering
		\includegraphics[scale=0.53]{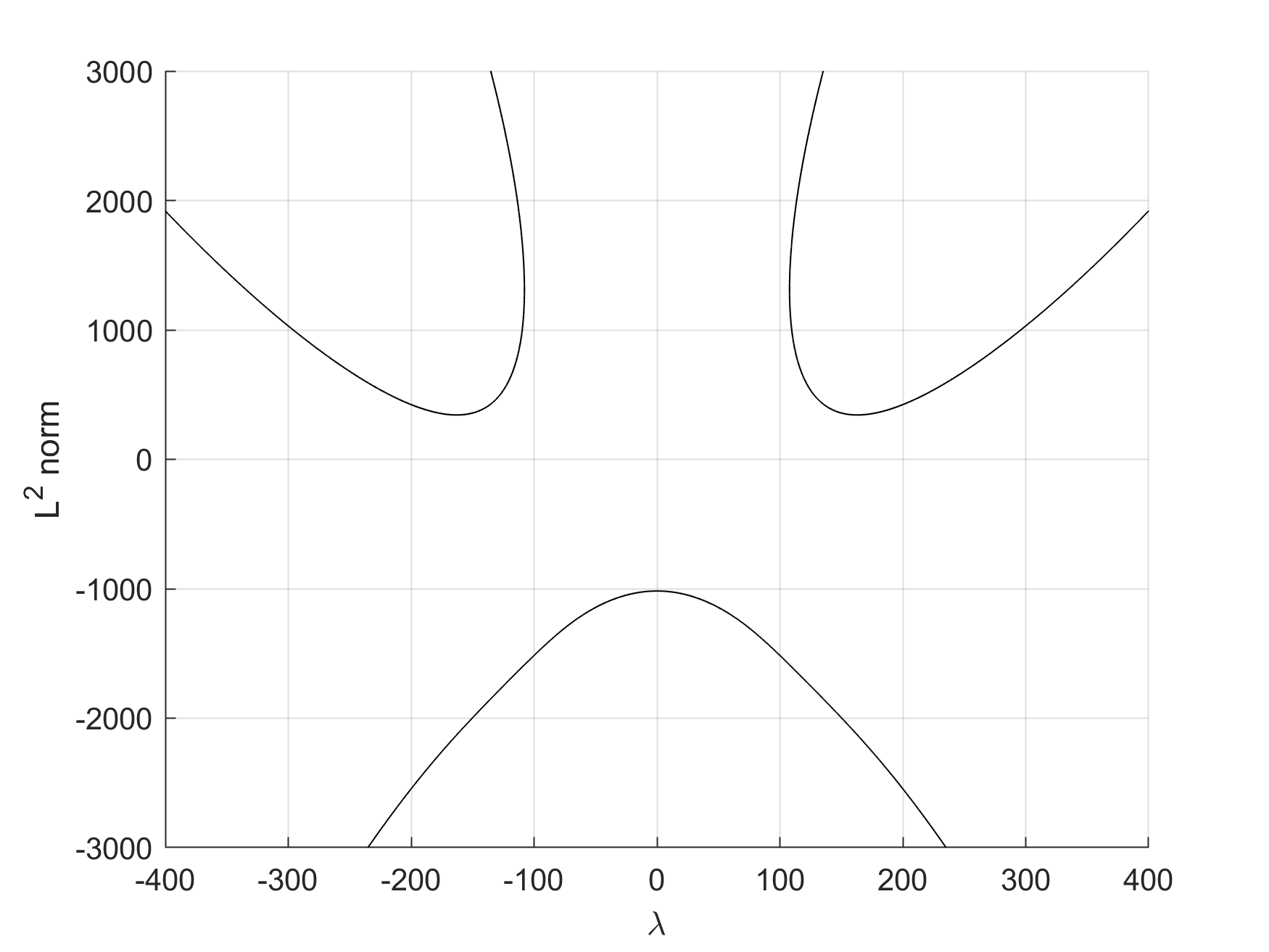}
		\caption{$\mu = 140$}\label{Fig6d}
	\end{subfigure}
\caption{Four representative bifurcation diagrams of 2-node solutions.}
\label{Fig6}
\end{figure}
More precisely, at $\mu=105$ the problem \eqref{1.1} possesses three components
of solutions with two interior nodes. Two of them bifurcating from $u=0$ at $\l_{-[1,3]}(105)$ and $\l_{+[1,3]}(105)$, respectively, and the third one linking $(\l_{-[2,3]}(105),0)$ with   $(\l_{+[2,3]}(105),0)$. According to
our numerical experiments these components are unbounded in $\R\times \mathcal{C}[0,1]$, and
are persistent for all further value of $\mu$ bellow some critical value, $\mu_c<108.1$, where the three components meet. Thus, for $\mu=\mu_c$ there is a component of the set of non-trivial solutions of
\eqref{1.1} bifurcating from $u=0$ at four different values of $\l$: $\l_{\pm[1,3]}(\mu_c)$ and $\l_{\pm[2,3]}(\mu_c)$. The plot in Figure \ref{Fig6b} shows the corresponding global bifurcation diagram for $\mu=108.1$, a value of $\mu$ slightly greater than $\mu_c$, where the three components of set of non-trivial solutions are very close. By comparison with the global bifurcation diagram for $\mu=105$, it becomes apparent that a global imperfect bifurcation phenomenon has happened at the critical value $\mu_c$. As a consequence of this imperfect bifurcation one of the components bifurcating from $u=0$ links
$(\l_{-[1,3]}(108.1),0)$ with  $(\l_{-[2,3]}(108.1),0)$, another links $(\l_{+[2,3]}(108.1),0)$ with   $(\l_{+[1,3]}(108.1),0)$, while the third one remains separated away from $u=0$. Actually, the latest one is separated away from zero for any further value of $\mu$. Therefore,
there have occurred a sort of reorganization in components of the set of 2-node solutions of
\eqref{1.1} as the parameter $\mu$ crossed the critical value $\mu_c$. The pictures in Figures \ref{Fig6c}, \ref{Fig6d} show the plots of the corresponding components for $\mu=110<\mu_3$ and $\mu=140>\mu_3$, where the previous bifurcations from $u=0$ of these components are lost. For larger values of $\mu$ the solutions along these three components become larger and larger  and it remains an open problem to ascertain whether, or not, \eqref{1.1} can admit some 2-node solution for sufficiently large $\mu$. Figure \ref{Fig7} shows the plots of some distinguished 2-node solutions of \eqref{1.1} along some of the curves of the bifurcation diagrams plotted in Figure \ref{Fig6}.
\par
\begin{figure}[h!]
	\centering
	\begin{subfigure}[t]{0.45\textwidth}
		\centering
		\includegraphics[scale=0.4]{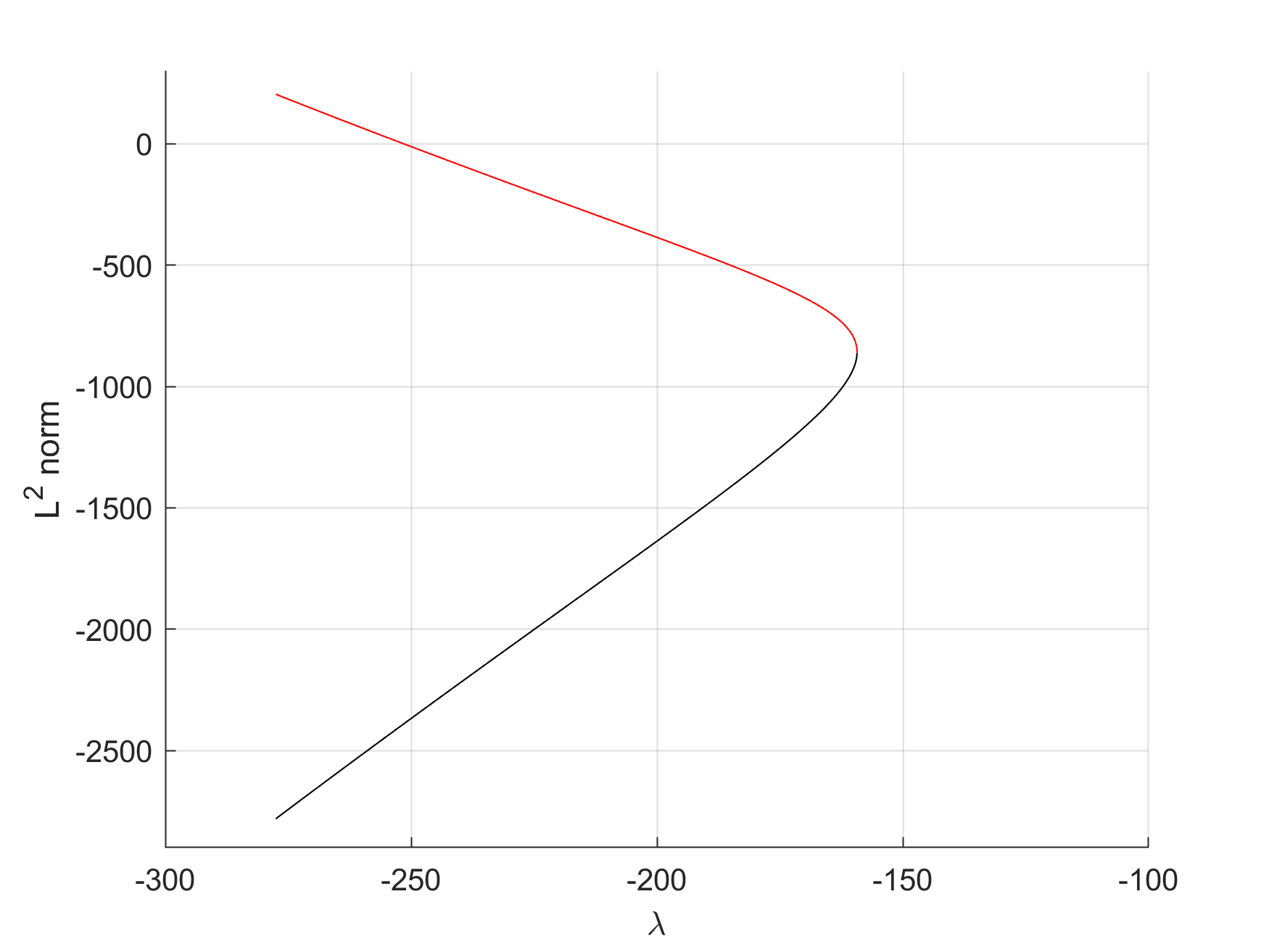}
		\caption{Left branch of  Figure \ref{Fig6a} ($\mu = 105$)}
	\end{subfigure}%
	\begin{subfigure}[t]{0.45\textwidth}
		\centering
		\includegraphics[scale=0.4]{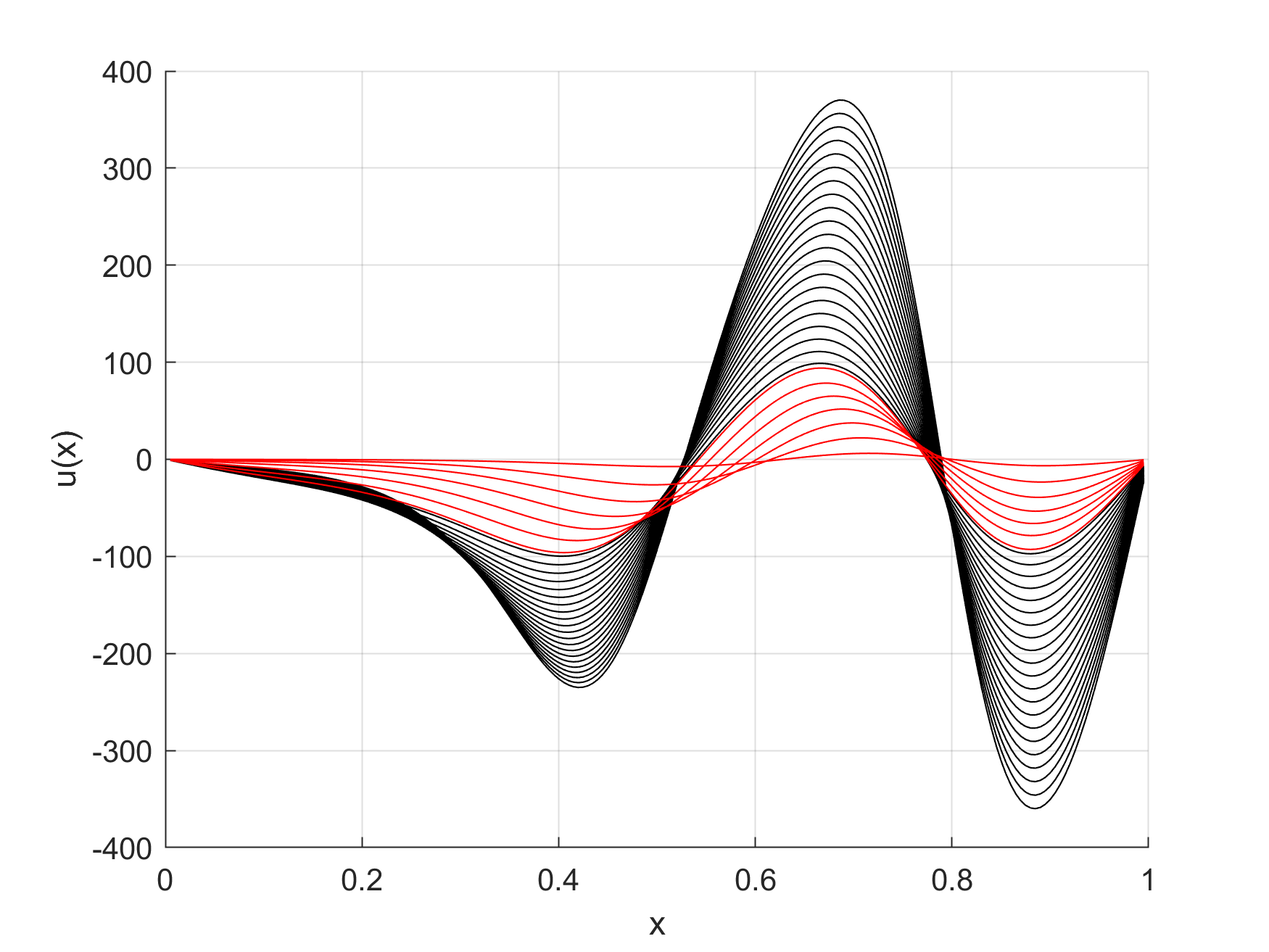}
		\caption{Corresponding profiles of solutions.}
	\end{subfigure}
	
	\begin{subfigure}[t]{0.45\textwidth}
		\centering
		\includegraphics[scale=0.4]{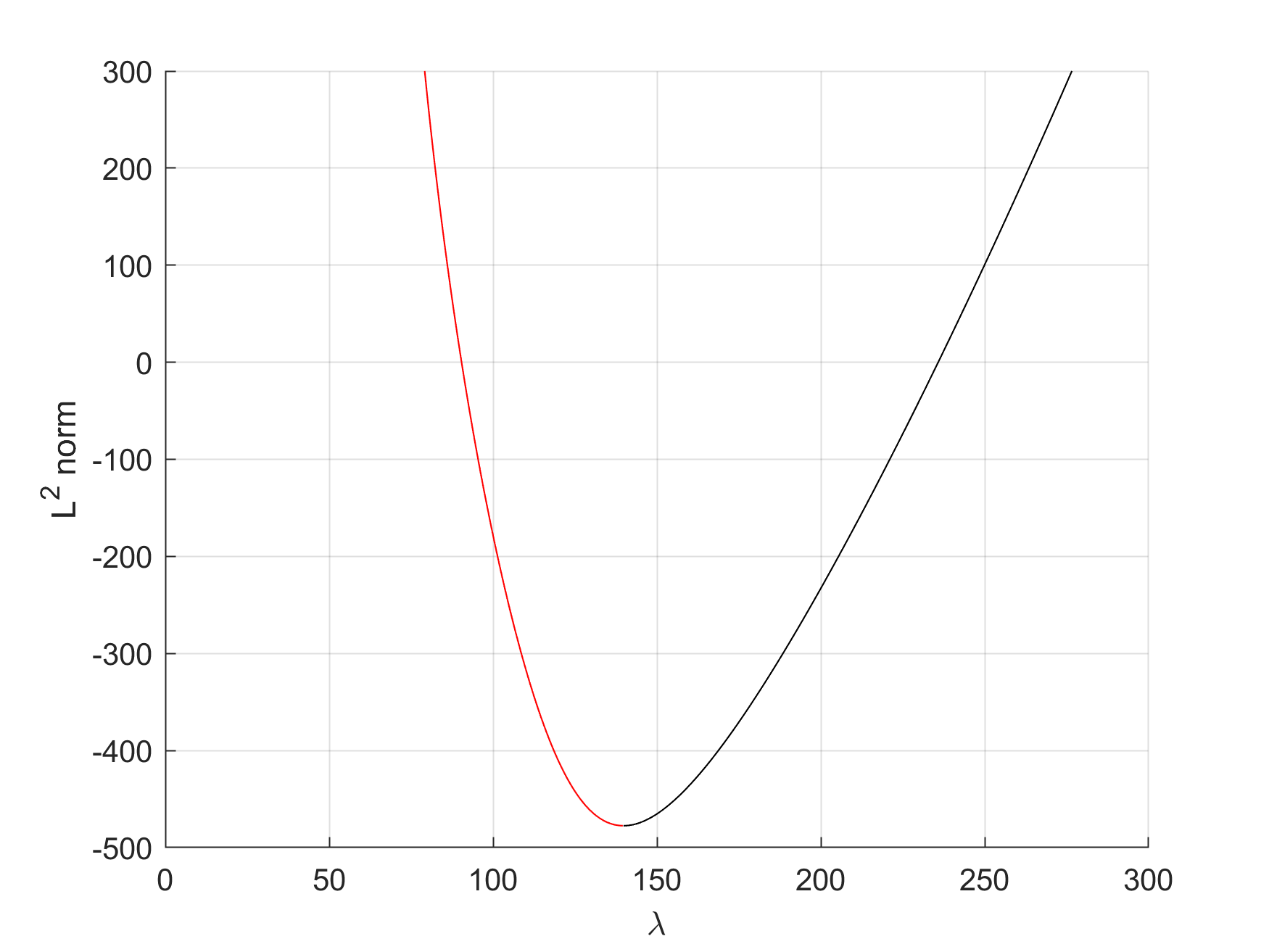}
		\caption{Right branch of Figure \ref{Fig6c} ($\mu = 110$) }
	\end{subfigure}%
	\begin{subfigure}[t]{0.45\textwidth}
		\centering
		\includegraphics[scale=0.4]{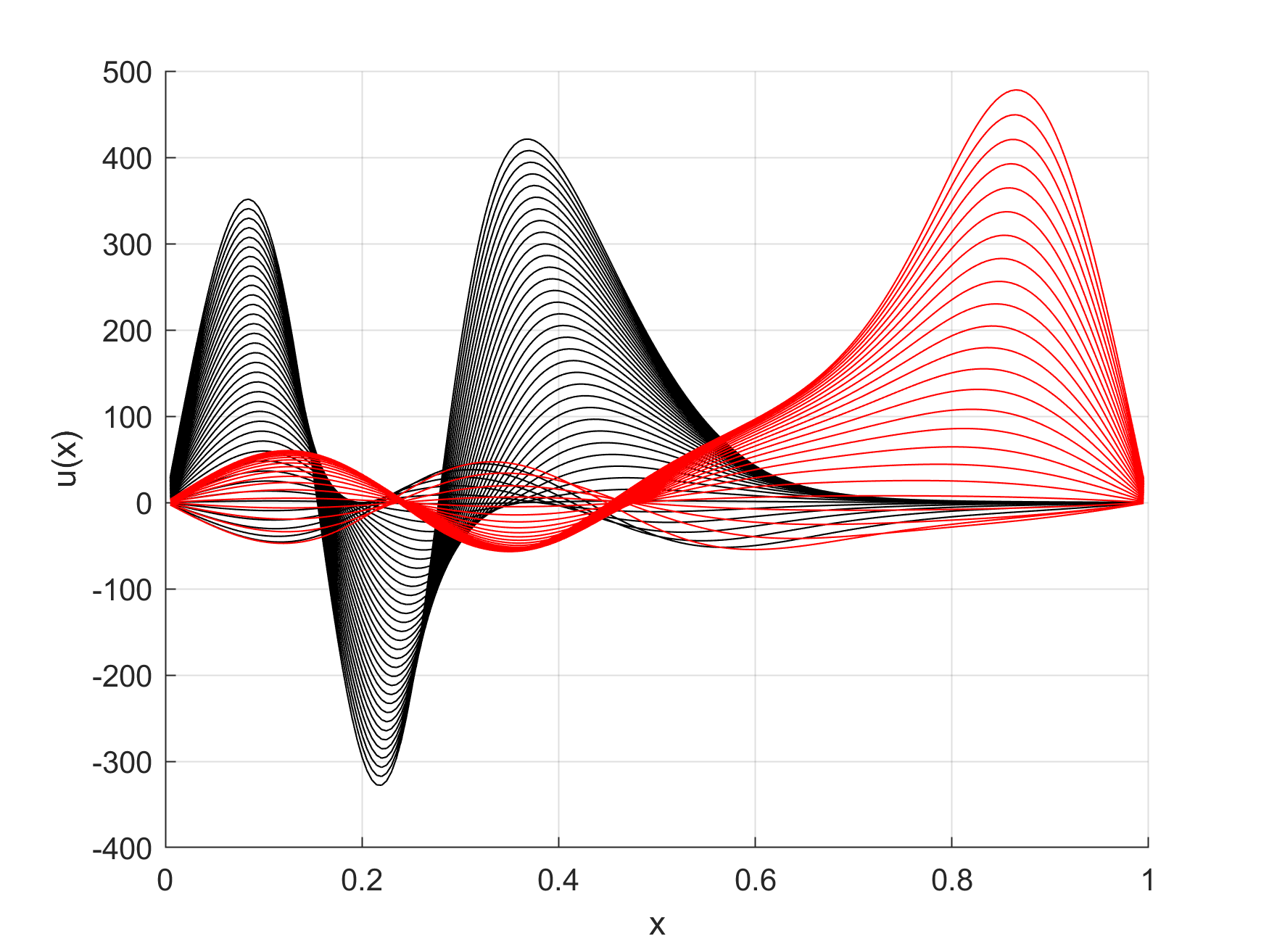}
		\caption{Corresponding profiles of solutions}
	\end{subfigure}
	\caption{Some plots of 2-node solutions (right) along the bifurcation diagrams of Figure \ref{Fig6} (left).}
	\label{Fig7}
\end{figure}
Finally, Figure \ref{Fig8} superimposes the global bifurcation diagrams of positive solutions found
in \cite{LGMM} (in blue) with the global bifurcation diagrams of nodal solutions with one node (in red) and two nodes (in black) computed in this paper for four different values of $\mu$: $0$, $54$, $70$ and $100$.
Although all the components of nodal solutions persist for these values of $\mu$, the component
of positive solutions shrinks to a single point and disappear at a value of $\mu$ above $54$ but very close to it. In Figure \ref{Fig8b} one can still see an small piece of blue trace component shortly before disappearing for an slightly grater value of $\mu$.
\begin{figure}[h!]
	\centering
	\begin{subfigure}[t]{0.45\textwidth}
		\centering
		\includegraphics[scale=0.53]{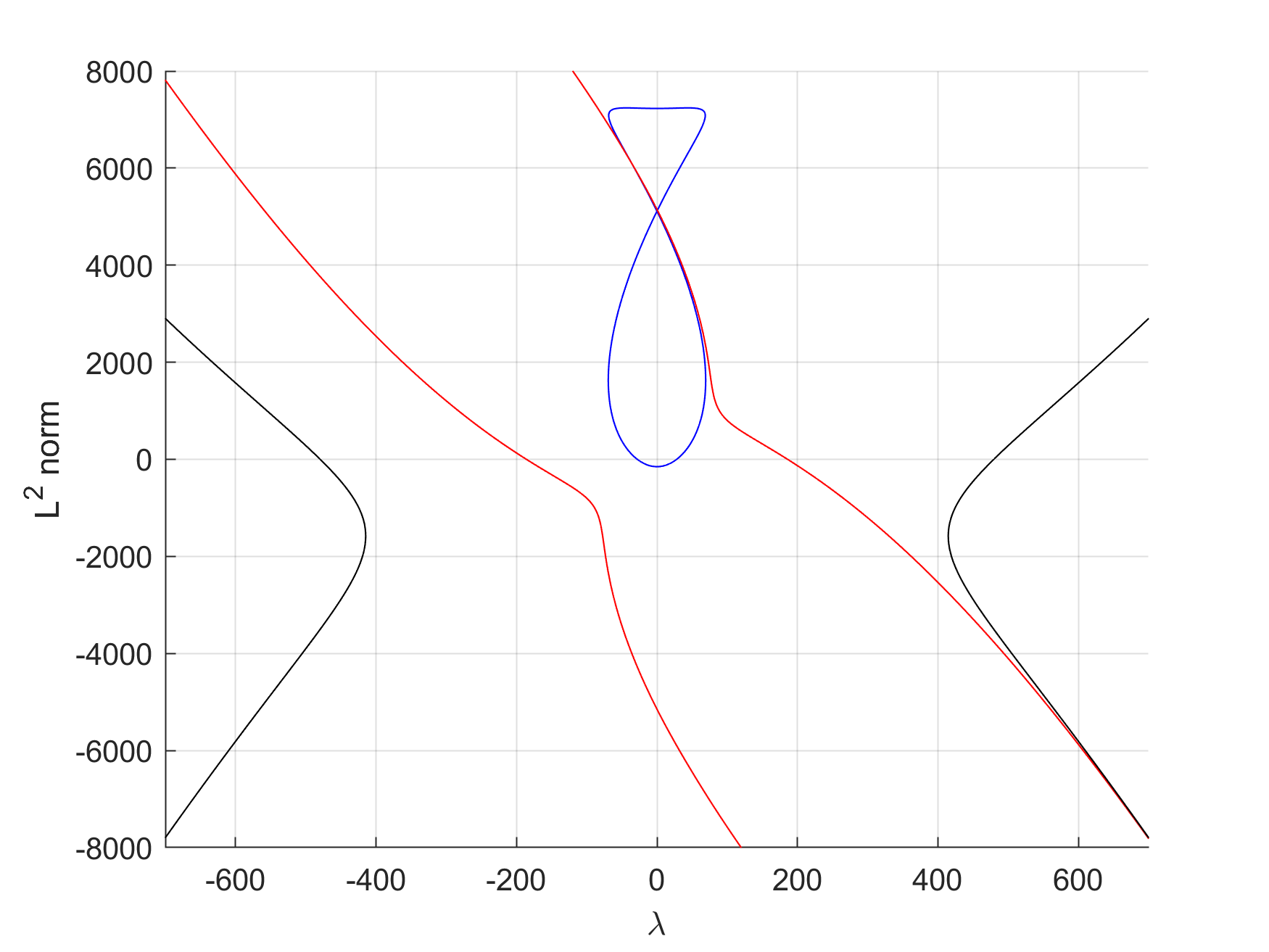}
		\caption{$\mu = 0$}\label{Fig8a}
	\end{subfigure}%
	\begin{subfigure}[t]{0.45\textwidth}
		\centering
		\includegraphics[scale=0.53]{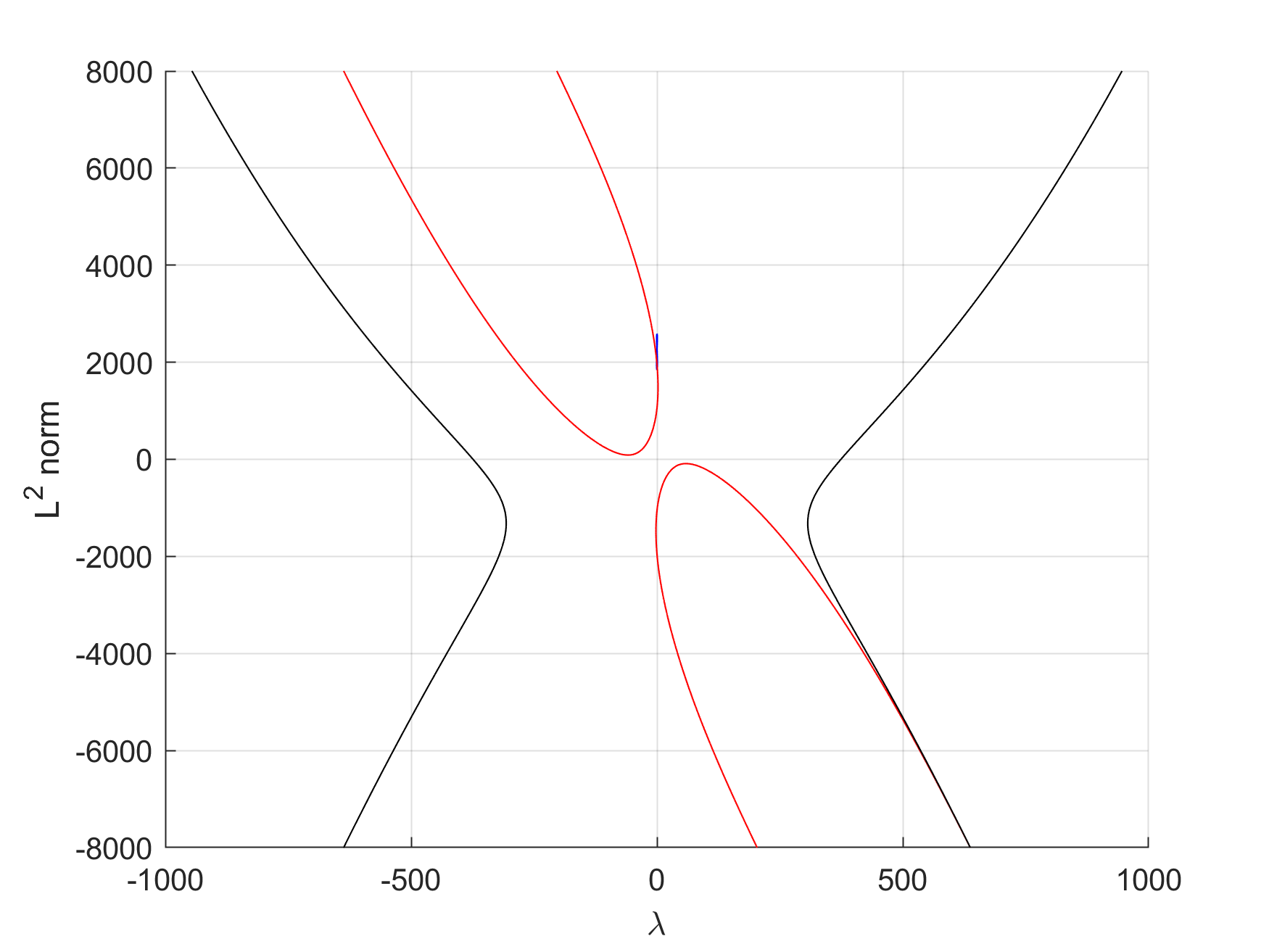}
		\caption{$\mu = 54$}\label{Fig8b}
	\end{subfigure}
	
	\begin{subfigure}[t]{0.45\textwidth}
		\centering
		\includegraphics[scale=0.53]{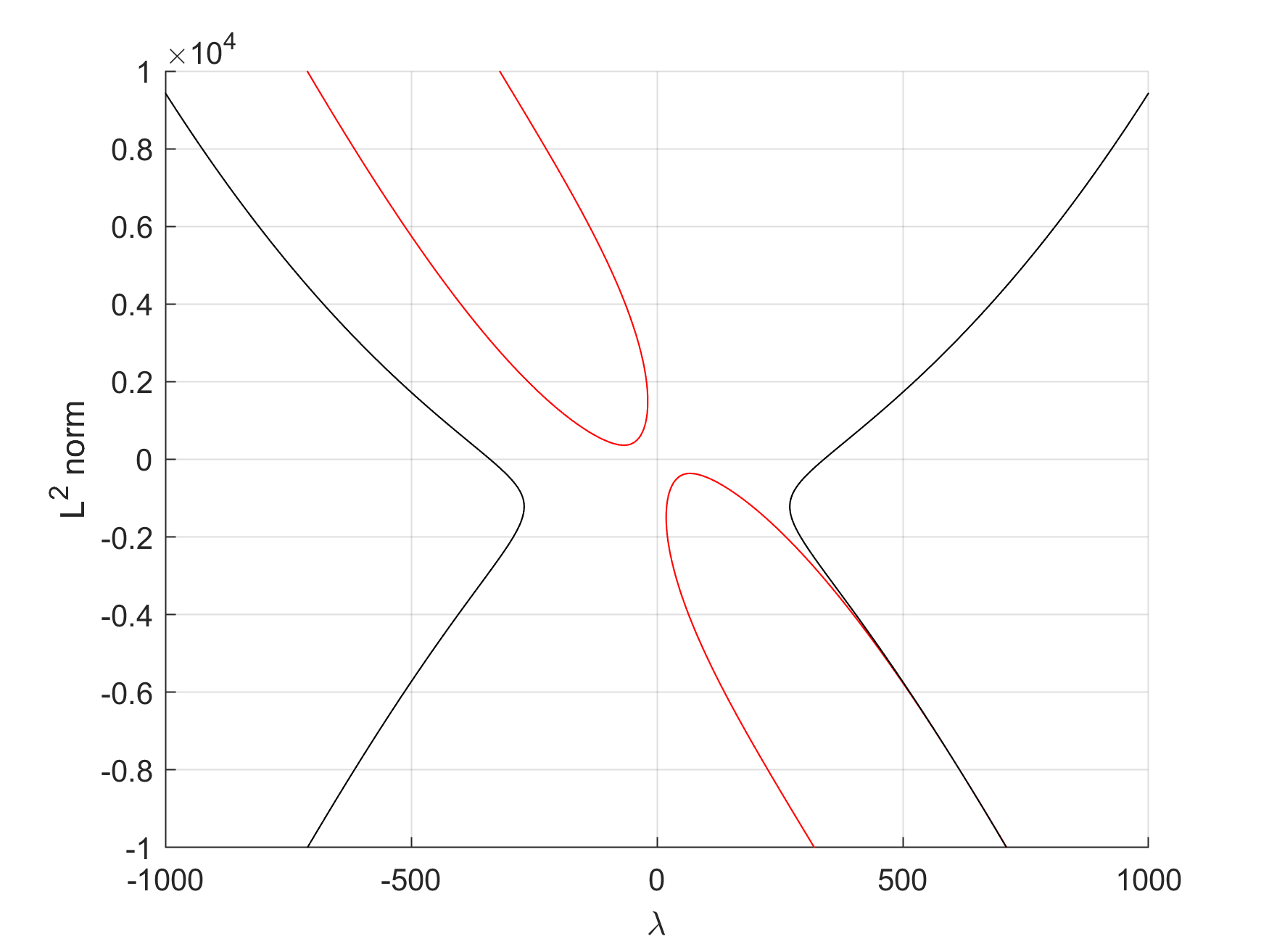}
		\caption{$\mu = 70$}\label{Fig8c}
	\end{subfigure}%
	\begin{subfigure}[t]{0.45\textwidth}
		\centering
		\includegraphics[scale=0.53]{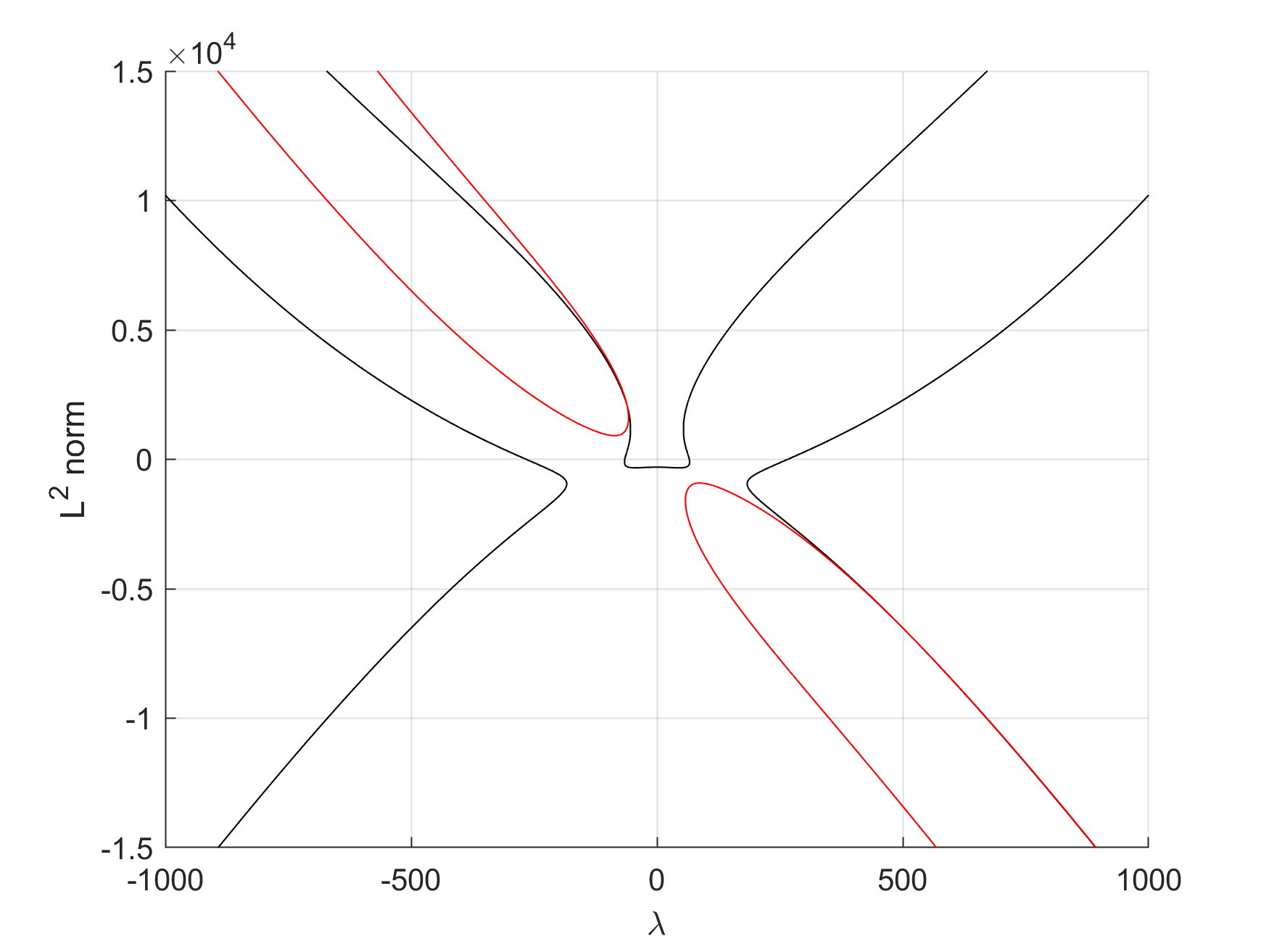}
		\caption{$\mu = 100$}\label{Fig8d}
	\end{subfigure}
	\caption{Some bifurcation diagrams with superimposed branches of positive (blue), 1-node (red) and 2-node (black) solutions.}
	\label{Fig8}
\end{figure}

\section{Numeric of bifurcation problems}

To discretize \eqref{1.1} we have used two methods. To compute the small solutions bifurcating
from $u=0$ we implemented a pseudo-spectral method combining a trigonometric
spectral method with collocation at equidistant points, as in most of our previous numerical experiments (see, e.g., \cite{GRLGJDE,GRLGNA,LGEDMM,LGMMJDE,LGMMTPB,LGMMMCS,LGMMT}). This gives high accuracy
(see, e.g., Canuto, Hussaini, Quarteroni and Zang \cite{CHQZ}). However, to  compute the large solutions we have used a centered finite differences scheme, which gives high accuracy at a lower computational cost, for  as it  provides us  with a much faster code to compute large pieces of curves of the global bifurcation diagrams.
\par
The pseudo-spectral method is easier to use and more efficient for choosing the shot direction from the trivial solution in order to compute the small nodal solutions of \eqref{1.1}, as well as to detect bifurcation points along the bifurcation diagrams. Its main advantage for accomplishing this task relies
on the fact that it provides us with the true bifurcation values from the trivial solution, while the differences scheme only provides  with an approximation to these  bifurcation values.
\par
For general Galerkin approximations, the local convergence of the solution paths at regular, turning and simple bifurcation points was proven by Brezzi, Rappaz and Raviart in \cite{BRR1,BRR2,BRR3} and by L\'{o}pez-G\'{o}mez et al. in \cite{LGEDMM,LGMMV} at codimension two singularities in the context
 of systems. In these situations, the local structure of the solution sets for the continuous and
 the discrete models are known to be equivalent. The global
continuation solvers used to compute the solution curves of this papers, as well as the
dimensions of the unstable manifolds of all the solutions along them, have been built from the
theory on continuation methods of Allgower and Georg \cite{AG}, Crouzeix and Rappaz \cite{CrRa},
Eilbeck \cite{Ei}, Keller \cite{Ke}, L\'{o}pez-G\'{o}mez \cite{LG88} and L\'{o}pez-G\'{o}mez, Eilbeck,
Duncan and Molina-Meyer \cite{LGEDMM}.
\par
The complexity of the bifurcation diagrams, as well as their quantitative features,
required an extremely careful control of all the steps in subroutines. This explains why
the available commercial  bifurcation packages, such as AUTO-07P are un-useful to deal with
differential equations with heterogeneous coefficients. As a matter of fact,
Doedel and Oldeman  admitted in \cite[p.18]{DO} that
\vspace{0.2cm}
\par
\begin{small}
\lq \lq Note that, given the non-adaptive spatial discretization,
the computational procedure here is not appropriate for PDEs with solutions that rapidly
vary in space, and care must be taken to recognize spurious solutions and bifurcations.\rq\rq
\end{small}
\par
\vspace{0.2cm}
This is just one of the main problems that we found in our numerical experiments, as the number of critical points of the solutions increases according to the dimensions
of unstable manifolds, and the turning and bifurcation points might be very close.

\end{document}